\documentclass[12pt]{amsart}
\usepackage{amsmath,amscd,amsthm,amsfonts, amssymb,amsxtra, mathrsfs,enumerate}
\usepackage{calrsfs}
\usepackage{caption}
\usepackage[us,12hr]{datetime}
\usepackage[all]{xy} \SelectTips{cm}{}
\usepackage{hyperref}
 \hypersetup{colorlinks=true,citecolor=blue}
   \usepackage{graphicx}
     \usepackage[margin=1.5in]{geometry}
   \usepackage[font=scriptsize]{caption}
\usepackage{todonotes}
\usepackage{tikz}
\usetikzlibrary{positioning}
\CDat
\usepackage{scalerel}[2016/12/29]
\newcommand{\B}[2]{\scaleobj{#2}{#1}}

\dedicatory{In memory of Bill Browder}
\newtheorem{thm}{Theorem}[section]  
\newtheorem*{un-no-thm}{Theorem}
\newtheorem{cor}[thm]{Corollary}     
\newtheorem{lem}[thm]{Lemma}         
\newtheorem{prop}[thm]{Proposition}

\newtheorem{bigthm}{Theorem}

\newtheorem{bigadd}[bigthm]{Addendum}

\theoremstyle{definition}
\newtheorem{defn}[thm]{Definition}   

\theoremstyle{definition}

\theoremstyle{definition}

\theoremstyle{remark}
\newtheorem{rem}[thm]{Remark}        

\newtheorem*{acks}{Acknowledgements}
\newtheorem*{out}{Outline}

\newtheorem{ex}[thm]{Example}

\newtheorem*{notation}{Notation}

\DeclareMathOperator{\Top}{Top}
\DeclareMathOperator{\Sp}{Spectra}
\DeclareMathOperator{\NP}{\it N\! P}
\DeclareMathOperator{\tr}{tr}

\begin{document}
\title{On Poincar\'e surgery}
\date{\today}
\author{John R.\ Klein}
\address{Wayne State University, Detroit, MI 48202}
\email{klein@math.wayne.edu}
\begin{abstract} We exhibit a homotopy theoretic proof of the Fundamental Theorem of Poincar\'e surgery
in the simply connected case. We also deduce the
Poincar\'e transversality exact sequence.
\end{abstract}
\maketitle
\setlength{\parindent}{15pt}
\setlength{\parskip}{1pt plus 0pt minus 1pt}

\def\Sp{\text{\bf Sp}}
\def\vo{\varOmega}
\def\vs{\varSigma}
\def\smsh{\wedge}
\def\flush{\flushpar}
\def\id{\text{id}}
\def\dbslash{/\!\! /}
\def\codim{\text{\rm codim\,}}
\def\:{\colon}
\def\holim{\text{holim\,}}
\def\hocolim{\text{hocolim\,}}
\def\Bbb{\mathbb}
\def\bold{\mathbf}
\def\Aut{\text{\rm Aut}}
\def\cal{\mathcal}
\def\sec{\text{\rm sec}}
\def\gda{G\text{\rm -}\delta\text{\rm -}\alpha}
\def\PDD{\text{\rm pd\,}}
\def\PD{\text{\rm P}}
\def\stableto {\,\, \mapstochar \!\!\to}
\def\tr{\text{\rm tr}}

\setcounter{tocdepth}{1}
\tableofcontents
\addcontentsline{file}{sec_unit}{entry}

\section{Introduction}

Poincar\'e surgery was a development of the  early 1970s  that attempted to transport the main results
of the surgery classification of manifolds into the category of Poincar\'e duality spaces.
The program was motivated by the observation that surgery on a framed embedded sphere
in a smooth manifold has a Poincar\'e duality space correlate, and the passage from geometry to algebra
passes through Poincar\'e duality spaces \cite[\S6]{Klein_survey}.

The  main assertions of Poincar\'e surgery highlight the extent to which manifold classification problems are reducible to homotopy theory.
By the Spring of 1969, Bill Browder had outlined the program. He remarked that a ``homotopy theoretical
problem should have a homotopy theoretical solution'' \cite[p.~1]{Quinn_mimeo}. 

 The goal of this paper is to provide a homotopy-theoretic, manifold-theory free, route to some of the main results
 of Poincar\'e surgery, primarily in the simply connected case. 

\subsection{The Fundamental Theorem} A basic notion in manifold surgery is that of a {\it normal map}. Let
$X$ be a connected Poincar\'e duality space of dimension $d$  and let $\xi$ be a stable vector bundle that
lifts the Spivak normal fibration \cite{Spivak}.
A normal map is a pair $(f,b)$ consisting of  a degree one map $f\: M \to X$ 
from a smooth $d$-manifold together with a stable bundle isomorphism $b\: \nu_M \cong f^\ast \xi$, where $\nu_M$ is the stable normal
bundle of $M$.  When $d\ge 5$, surgery theory provides an obstruction $\sigma(f,b) \in L_d(\pi,w) $ to deciding when $(f,b)$ is cobordant to 
a homotopy equivalence, where $ L_d(\pi,w)$ is the quadratic $L$-group in degree $d$ with 
$\pi = \pi_1(X)$ and $w= w_1(\xi)$ the first Stiefel-Whitney class.

There is a Poincar\'e duality space variant of the above called  a {\it Poincar\'e normal map}.  In this instance, the data
$(f,b)$ consist of a degree one map of Poincar\'e duality spaces $f\: M\to X$ and a stable fiber homotopy equivalence $b\: \xi_M \simeq f^\ast \xi$, where $\xi_M$ is the 
Spivak normal fibration of $M$ and $\xi = \xi_X$ is the Spivak normal fibration of $X$.

Let 
\[
\cal N(X,\xi) 
\]
be the set of bordism classes of Poincar\'e normal maps with target $X$.
The surgery obstruction in this setting is well-defined and is a function
\begin{equation} \label{eqn:surgery-obstruction}
\sigma\: \cal N(X,\xi) \to L_d(\pi,w) 
\end{equation}
(see \S\ref{sec:bordism} below). If $(f,b)$ is normally cobordant to a homotopy equivalence, then $\sigma(f,b) = 0$.
We will prove the following partial converse in the simply connected case:


 
 \begin{bigthm}[``Fundamental Theorem''] \label{bigthm:fund-thm} Let $(f,b)\: M^d \to X^d$ be a Poincar\'e normal map.
Assume $d\ge 7$ or $d=5$. Suppose that $\pi$ is trivial. If $\sigma(f,b) = 0$, then $(f,b)$ is normally cobordant to a  homotopy equivalence.
  \end{bigthm}

\noindent (Compare \cite[cor.~1.4]{Quinn}, \cite[thm.~3.4, cor.~6.2]{Jones}, \cite[thm.~4.1]{Hodgson_surgery},
 \cite[thm.~5.1]{HV}.) 
 We also prove the following companion result, which does not require the triviality of $\pi$.

\begin{bigthm}[``Absolute Wall Realization''] \label{bigthm:sigma-onto} If $ d \ge 8$ is even, then $\sigma$ is  onto.
\end{bigthm}

\subsection{Transversality}
We now let $X$ denote a connected topological space of finite type equipped with a stable vector bundle $\xi$ of virtual dimension zero.
Let $\Omega_d^{\text{diff}}(X,\xi) $ denote the bordism group 
 generated by  maps $f\: P \to X$ in which $P$ is a closed smooth manifold of dimension $d$, and
where the pullback of $f^\ast\xi$ is identified with the stable normal bundle of $P$. The Pontryagin-Thom construction defines a homomorphism
\begin{equation} \label{eqn:smooth-bordism-to-homotopy}
\Omega_d^{\text{diff}}(X,\xi) \to \pi_{d}(X^\xi)
\end{equation}
where the target denotes the homotopy group in degree $d$ of the Thom spectrum $X^\xi$. 
By transversality, \eqref{eqn:smooth-bordism-to-homotopy} is an isomorphism \cite[thm.~II.2.1]{Browder}.

Now suppose instead that $\xi$ is a stable spherical fibration over $X$.  
Then similarly, one has a Pontryagin-Thom homomorphism
\begin{equation}\label{eqn:bordism-to-homotopy}
c\: \Omega_d^{P}(X,\xi) \to \pi_{d}(X^\xi)
\end{equation}
whose source is the bordism group  generated by  $f\: P \to X$, in which $P$ is
a Poincar\'e duality space, and $f^\ast\xi$ is identified with the Spivak normal fibration of $P$.

\begin{ex} Let $X =\ast$ be a point.
A construction of Milnor  yields a non-smoothable, stably parallelizable piecewise linear manifold $P$ of dimension $4k$
with signature $8$  \cite{Milnor-diff-top}. In particular, the Spivak normal fibration of $P$ is trivializable.
The signature defines a homomorphism $\Omega_{4k}^{P}(\ast,\xi) \to \Bbb Z$, and it follows that
the bordism class of $P$ is a non-torsion element. However,
 $\pi_{4k}(\ast^\xi) = \pi_{4k}(S^0)$ is finite \cite{Serre}.
 It follows that the kernel of  \eqref{eqn:bordism-to-homotopy} has elements of infinite order in this case.
\end{ex}

Under certain restrictions, 
 we will see that the failure of \eqref{eqn:bordism-to-homotopy} to be an isomorphism is detected
by the surgery obstruction groups.  Set $\pi = \pi_1(X)$ and $w = w_1(\xi)$.

\begin{bigthm}
\label{bigthm:stronger-exactness} 
There is a long exact sequence
of abelian groups
\[
\cdots \to \cal F_d(X,\xi)  \to \Omega_d^{P}(X,\xi) @> c >> \pi_{d}(X^\xi) \to  \cal F_{d-1}(X,\xi) \to \cdots \, .
\]
and a natural homomorphism 
\[
\sigma'\: \cal F_d(X,\xi) \to L_d(\pi,w)\, .
\]
If $d \ge 7$ and $\pi$ is trivial, then $\sigma'$ is an isomorphism.
\end{bigthm}
\noindent (See \ref{defn:Fd} for the definition of $\cal F_d(X;\xi)$.)
\medskip

Our method of proof establishes a bit more:

\begin{bigadd} \label{bigadd:stronger-exactness} Assume $d \ge 7$.
 Then  $\sigma'$ is surjective if $\pi$ is non-trivial and $d$ is even. 
If Theorem \ref{bigthm:fund-thm} holds for non-trivial $\pi$, then $\sigma'$ is injective. 
\end{bigadd}

\subsection{Historical Background} 
A 1972  AMS  Bulletin article by Frank Quinn announced  a homotopical solution to Poincar\'e surgery \cite{Quinn}. 
The announcement was a distillation of the results of the circulating unpublished manuscript \cite{Quinn_mimeo}.
Unfortunately, a crucial step in Quinn's approach relied on certain statements about desuspending cofibrations
whose proofs were soon thereafter discovered to contain serious gaps  \cite[lem.~1.5]{Quinn}, \cite[cor.~1.3]{Quinn_cofibration}.
 In fact, the statements are false, although a counterexample to them did not appear in the literature until 1993 \cite{Hutt}.
In summary, early attempts at getting the program to succeed faced challenges that were never completely overcome.

Around the same time as Quinn's article, the work of  Levitt  \cite{Levitt_thom}, \cite{Levitt_cobordism} on Poincar\'e transversality and Jones \cite{Jones} on patch spaces
appeared. These works approached the subject along different lines, each involving significant input from manifold theory.
Hodgson's 1974 paper \cite{Hodgson1974} on general position makes sophisticated
use of manifolds and is related to  Levitt's approach. Hodgson
subsequently extended \cite{Hodgson1974} in the 1984 paper \cite{Hodgson_surgery} to deduce the Fundamental Theorem.
The 1993 book by Hausmann and Vogel appeared to surmount many of the difficulties found in the earlier attempts \cite{HV}.
However, Hausmann and Vogel also imported  techniques from the theory of smooth manifolds and their methods
involve verifying a procession of nested double inductions.

In summary, none of the above  works settled the main issues in a way one might hope: 
One was left wondering whether or not the statements of Poincar\'e surgery, if true, are really theorems
of manifold theory rather than of homotopy theory.

\subsection{Remarks on our approach}
We will prove Theorem \ref{bigthm:fund-thm} by following the path of Browder's book \cite[chap.~IV]{Browder} while  inserting along the way  
the Poincar\'e embedding results of the  author \cite{Klein_haef}, \cite{Klein_compression}, \cite{Klein_sphere} 
in place of the Whitney embedding theorems. The failure of these results 
to address the case of embeddings of $3$-spheres in
Poincar\'e spaces of dimension 6 is the
reason why we must assume $d \ne 6$.
 
We are optimistic that the Poincar\'e embedding results of \cite{Klein_compression} are sufficient for
proving Theorem \ref{bigthm:fund-thm} in the case of general fundamental groups, but this will require a higher level
of refinement in passing from the geometry to the algebra. 

Our proof of Theorem \ref{bigthm:stronger-exactness} also relies on the author's earlier Poincar\'e embedding results and
 on  Theorems \ref{bigthm:fund-thm} and  \ref{bigthm:sigma-onto}.  It is manifold-free when $\pi$ is trivial. Our proof  extends to the non-simply connected
case provided that Theorem \ref{bigthm:fund-thm} holds for all fundamental groups and
provided Theorem \ref{bigthm:sigma-onto}  holds when $d$ is odd. We plan to thoroughly address the non-simply connected case in another paper.

\begin{out} Section \ref{sec:prelim} contains preliminary material. Section \ref{sec:easy} discusses variants of Poincar\'e and normal space
 bordism. 
Section \ref{sec:Quinn} sketches the definition of the surgery obstruction of a (normal, Poincar\'e) pair.  
Section \ref{sec:normal-maps} introduces four variations on the notion of a Poincar\'e normal map;
we show that these variants coincide after passing to the corresponding normal bordism sets. 
The main effort of the paper appears in section \ref{sec:bordism}: First we sketch Ranicki's
definition of the surgery obstruction of a normal map  and show that it factors through $\cal F_d(X,\xi)$. 
We then prove Theorem \ref{bigthm:sigma-onto}, relying on the material of later sections \ref{sec:toolbox} and \ref{sec:realization}
and thereafter give criteria for deciding when an element of $\cal F_d(X,\xi)$ arises from a normal map.\footnote{The reason for postponing some of the material 
to later sections is to emphasize the main ideas at the expense of the technical details appearing later.}
 In section 
\ref{sec:exact-sequence} we prove 
Theorem \ref{bigthm:stronger-exactness}. 
 Section \ref{sec:toolbox} is a toolbox of results used in Poincar\'e surgery.
In section \ref{sec:d-odd} we prove Theorem \ref{bigthm:fund-thm} for $d$ odd, and in section \ref{sec:d-even} we prove it for $d$ even.
Section \ref{sec:realization} concerns Wall realization 
for even dimensional $L$-groups in the setting of Poincar\'e spaces. 
Section  \ref{sec:thick} is a sketch of  the theory of Poincar\'e space thickenings in the stable range.
\end{out}

\begin{acks} This project was initiated by Andrew Ranicki's encouragement and began to take shape
during a one month visit I made to  the University of Oslo in May, 2008 during which I had intense discussions
about Poincar\'e transversality
with Bjørn Jahren to whom I am grateful. Little bursts of activity occurred here and there until 2011. 
 The interval 2011-2023 was a period of cryogenic suspension. 

The project was reanimated after I
spoke with Shmuel Weinberger at a Banff Workshop in December, 2023.  
I am indebted to Tyler Lawson for help with the proof of  Lemma \ref{lem:unstable-stable-refine}.
 Bruce Williams provided me with help on Wall's realization theorem. 
 
I owe the referee a debt of gratitude for a careful and generous reading of the manuscript; 
the paper is considerably clearer for the referee's many suggestions.
\end{acks}

\section{Preliminaries} \label{sec:prelim}
\subsection{Spaces over $X$}
Let $\Top$ denote the category of compactly generated spaces, equipped with
the Quillen model structure in which a map $X\to Y$ is a weak equivalence (resp.~fibration)
if it is a weak homotopy equivalence (resp.~Serre fibration). 
A map $X\to Y$ is a cofibration if it is the result of attaching cells to $X$, or if it is the retract of such a map.
We will implicitly work with cofibrant spaces and apply cofibrant approximation whenever necessary to guarantee
cofibrancy.

A non-empty space $X$ 
is {\it $0$-connected} if it is path connected.
It is $r$-connected with  $r > 0$, if it is $0$-connected and the set of homotopy classes
$[S^k,X]$ has one element for $k\le r$.
A map of non-empty (possibly unbased) spaces is
{\it $r$-connected} if its homotopy fibers, taken with respect to
all choices of basepoint, is $(r{-}1)$-connected. Let $I = [0,1]$ denote the unit interval.

A commutative square of spaces
\[
\xymatrix{
A \ar[r] \ar[d] & C \ar[d] \\
B \ar[r] & D
}
\] 
is a {\it homotopy pushout} if the induced map
\[
B\times \{0\} \cup A \times I \cup C\times \{1\} \to D
\]
is a weak equivalence, where the source is the double mapping cylinder of the diagram $B \leftarrow A \to C$.

A space $X$ is {\it homotopy finite} if it has the homotopy type of a finite CW complex.
It is of {\it finite type} if it has the homotopy type of a CW complex having finitely many cells in each dimension.
A {\it $k$-skeleton} for a space $X$  is a $k$-connected map of spaces $Y\to X$ in which $Y$ is a CW complex of dimension $\le k$.

A {\it stable map} $X \to  Y$ of based spaces is a map of suspension spectra $\Sigma^\infty X\to \Sigma^\infty Y$.
The abelian group of homotopy classes of stable maps $X\to Y$ is denoted by $\{X,Y\}$.
If $X$ is homotopy finite, then a stable map is represented by a map of based spaces
 $\Sigma^j X \to \Sigma^j Y$ for some $j \ge 0$, where $\Sigma^j X$ is the reduced $j$-fold suspension of $X$. 
We identify two such maps if they agree after some iterated suspension.

Let $\Top_{/X}$ be the category of spaces over $X$, whose objects
are spaces $Y$ equipped with (structure) map $Y\to X$ and a morphism is a map
of spaces which is compatible with the structure map. Let $\Top_{/X} \to \Top$
be the forgetful functor. We equip $\Top_{/X}$ with the induced model structure
in which a morphism is a weak equivalence (resp.~fibration) if it is so when considered in $\Top$.
A morphism is a cofibration if it possesses the left lifting property with respect to the acyclic fibrations.
It can be shown that an object $Y$ is cofibrant if it is obtained from the initial object by attaching
cells, or if it is a retract thereof. By cofibrant approximation, any object is weakly equivalent to a cofibrant one, and
we will apply cofibrant approximation whenever required.
A morphism $Y\to Z$ of $\Top_{/X}$ is $k$-connected if it is when considered in $\Top$.
In particular, an object $Y$ is $k$-connected if and only if its structure map $Y\to X$ is $(k+1)$-connected.

The {\it unreduced fiberwise suspension}
is the functor $S_X \: \Top_{/X} \to \Top_{/X}$ which assigns to
an object $Y$ the double mapping cylinder
\[
S_X Y = (X \times \{0\}) \cup Y\times I \cup (X \times \{1\})\, .
\]

\subsection{Stable spherical fibrations}
Let $BG$ be the classifying space of stable spherical fibrations \cite{Stasheff_BG}. 
This is the classifying space of the grouplike topological monoid $G$ consisting of
stable self equivalences of the zero sphere $S^0$.

By a {\it stable spherical fibration} over a space $X$, we mean a map
$\xi\: X\to BG$. In other words, it is an object
\[
X \in \Top_{/BG} \, .
\]
To relate this notion to the classical one, 
let $p\: EG\to BG$ be a universal classifying fibration in which $EG$ is a free contractible
$G$-space. Using $\xi$, we may form the fiber product
 $\tilde X = X \times_{BG} EG$; the latter is equipped with a free action of $G$. Let $S$ be the sphere spectrum. Then 
we obtain a parametrized spectrum
  \[
 \tilde X\times_G S \to X
 \]
 whose fibers are spherical.

The {\it Thom spectrum} $X^\xi$ of $\xi$ 
is the orbit
 spectrum
 \[
X^\xi := (\tilde X_+) \smsh_{G} S\, .
 \]
 \begin{notation}
If $Y \in \Top_{/X}$ is an object, we will typically abuse notation by denoting the induced
stable spherical fibration $Y \to X @>>> BG$ by $\xi$.
\end{notation}

The assignment $Y \mapsto Y^\xi$ defines a spectrum-valued functor
\[
(-)^\xi\: \Top_{/X} \to \Sp
\]
which gives rise to an unreduced homology theory on $\Top_{/X}$:
\[
\cal H_\ast(Y;\xi) := \pi_\ast(Y^\xi)\, .
\]

\subsection{Poincar\'e duality} 
A  space $P$ is a {\it Poincar\'e duality space} 
of  dimension $d$ if there exists a pair
\[
(L,[P])
\]
in which $L$ is a  bundle  of coefficients that is locally isomorphic
to $\Bbb Z$, and  $[P] \in H_d(P;L )$ is a homology class
such that the associated cap product homomorphism
\[
\cap [P]\:H^*(P;\cal B) \to H_{d{-}*}(P;\cal B \otimes L )
\]
is an isomorphism in all degrees for every local coefficient bundle $\cal B$ \cite{Wall_PD}.
If the pair $(L,[P])$ exists, then it
is defined up to unique isomorphism; $L$ is called
the {\it orientation sheaf} and $[P]$ the {\it fundamental class}.

Similarly, given a Poincar\'e space $\partial P$
of dimension $d-1$, one has the notion of {\it Poincar\'e pair} $(P,\partial P)$ of dimension $d$,
where now $[P] \in H_d(P,\partial P;L )$ induces an isomorphism
\[
\cap [P]\:H^*(P;\cal B) \to H_{d{-}*}(P;\partial P; \cal B\otimes L)\, 
\]
(cf.~\cite[thm.~B]{Klein-Qin-Su}).
In this paper, we will assume that Poincar\'e spaces (pairs)
are homotopy finite, i.e., they have the homotopy type of a finite CW complex
(resp.~finite CW pair).

\subsection{The Spivak normal fibration} Let $P$ be a Poincar\'e duality space of dimension $d$. 
Let $\xi\: P \to BG$  be a  stable spherical fibration.
An orientation sheaf $L$ on $P$ for $\xi$ is defined by the first Stiefel-Whitney class 
$w = w_1(\xi) \in H^1(P,\Bbb Z_2)$, where $\Bbb Z_n$ denotes the cyclic
group of order $n$. Note that  $w$ may be regarded as a homomorphism $\pi\to \Bbb Z_2$
called the {\it orientation character}, where $\pi$ is the fundamental groupoid of $P$.

The stable spherical fibration $\xi$ is  a {\it Spivak fibration} for $P$ if there exists a homotopy class $\alpha\in \pi_d(P^\xi)$
such that its Hurewicz image
\[
[P] := \alpha_\ast([S^d])\in H_{d}(P^\xi) \cong H_d(P; L)
\]
is a  fundamental class, where $[S^d] \in H_d(S^d;\Bbb Z)$ is the generator \cite{Spivak}, \cite{Browder_pd-surgery}, \cite{Klein_dualizing};
note that in the display we have used the Thom isomorphism. The pair $(\xi, \alpha)$ is unique up to fiber homotopy equivalence 
\cite[I.4.19]{Browder}, \cite{Klein_dualizing}.\footnote{In \cite{Klein_dualizing} we employ a different convention in which $\xi$ has virtual dimension $-d$.}
The element $\alpha\in \pi_d(P^\xi)$ is called the {\it normal invariant}. 

\begin{defn} The triple $(P^d,\xi,\alpha)$ is called a {\it Spivak normal structure}.
\end{defn}

If $(P,\partial P)$ is a Poincar\'e pair, then one has a Spivak fibration $\xi$,
which can be defined by pulling back the Spivak fibration of the double $P \cup_{\partial P} P$
along the left inclusion $P \subset P\cup_{\partial P} P$. 
In this instance, the normal invariant $\alpha$ is a  homotopy class of pairs
$(D^{n},S^{n-1}) \to  (P^\xi,(\partial P)^\xi)$ or equivalently, a homotopy class
\[
S^n \to P^\xi/(\partial P)^\xi\, .
\]

\section{Bordism Theories} \label{sec:easy}

Let  $\xi\: X\to BG$ be a stable spherical fibration, where $X$ is a space of finite type.
If $Q\in \Top_{/X}$ is an object, we remind the reader that $Q^\xi$ denotes the Thom spectrum
of the pullback of $\xi$ to $Q$.

\subsection{Poincar\'e and normal spaces over $X$}
A {\it normal structure over $X$}  of dimension $d$ is a pair $(Q,\alpha)$ in which
\begin{enumerate}[(i).]
\item
$Q \in \Top_{/X}$ is a homotopy finite object, and
\item $\alpha\in \pi_d(Q^{\xi})$ is a homotopy class.
\end{enumerate}
Informally, we say that $Q$ is a {\it normal space over $X$}. By analogy, we will also call $\alpha$ the normal invariant.

If in addition the Hurewicz image $[Q] := \alpha_\ast([S^d]) \in H_d(X^\xi)$ equips
$Q$ with the structure of a Poincar\'e space, then 
$(Q,\alpha)$ is a {\it Spivak normal structure over $X$} and  informally, 
we say that $Q$ is a Poincar\'e space over $X$.

\begin{ex} \label{ex:tautology}  There are two extreme cases of $X$ worth
singling out: The first case is when $X = Q$ and the second case is when $X = BG$.
In each case the data amount to a triple $(Q,\xi,\alpha)$ in which
$Q$ is a space, $\xi\: Q\to BG$ is a stable spherical fibration, and $\alpha\in \pi_d(Q^\xi)$.
This agrees with the notion of  normal space as defined by Quinn \cite{Quinn}. In the Poincar\'e case,
$(Q,\xi,\alpha)$ is the same thing as a Spivak normal structure.
\end{ex}

A {\it normal pair} of dimension $d+1$ over $X$ is a pair
\[
((Q,\partial Q),\alpha)\, ,
\]
in which $(Q,\partial Q) \in \Top_{/X}$ is a homotopy finite pair
over $X$, and
\[
\alpha\in \pi_d(Q^{\xi},(\partial Q)^{\xi})
\]
is a homotopy class.
In particular, the boundary $(\partial Q, \alpha_{|\partial Q})$ is a normal  structure of dimension $d$.
Observe that $\alpha$ is equivalent to specifying a homotopy class in $\pi_d(Q^{\xi}/(\partial Q)^{\xi})$.

A normal space of dimension $d$ over $X$ is said to {\it bound} if it is the boundary
of a normal pair over $X$ of dimension $d+1$. 
A normal pair $(Q,\partial Q)$ over $X$ of dimension $d+1$ is a  {\it (normal, Poincar\'e)} pair over $X$  if its boundary is Poincar\'e.
Similarly, it is a {\it Poincar\'e pair} over $X$ if $\alpha_\ast([S^{d+1}])$ and $w_1(\xi)$ equip  $(Q,\partial Q)$ with the structure of a Poincar\'e pair.

\subsection{Normal space bordism}
 Normal spaces $(Q_0,\alpha_0)$ and $(Q_1,\alpha_1)$ over $X$ of dimension $d$ are {\it bordant} if 
 \[
 (Q_0 \amalg Q_1, \alpha_0 - \alpha_1)
 \]
 bounds, where the normal invariant $\alpha_0 - \alpha_1$ is the homotopy class of the composition
\begin{equation} \label{eqn:pinch}
 S^d @> \text{pinch} >> S^d \vee S^d @>\alpha_0 \vee -\alpha_1>> Q_0^\xi \vee Q_1^\xi\, .
 \end{equation}

The set of such bordism classes forms an abelian group $\Omega_d^N(X,\xi)$.
The operation which sends a normal space $Q$ over $X$ to the composition
\[
S^d @>\alpha >> Q^\xi \to X^\xi
\] 
defines a homomorphism
\[
\phi\:  \Omega_d^N(X,\xi)\to \pi_d(X^\xi)\, .
\] 
 
 \begin{lem}[{cf.~\cite[prop.~3.9]{HV}}] \label{lem:phi}   The homomorphism $\phi$ is an isomorphism.
 \end{lem}
 
 \begin{proof}  We may assume that $X$ is a CW complex with a finite number of cells in each dimension. Let $Y$ be
 the $(d+1)$-skeleton of $X$. By the cellular approximation theorem, the homomorphisms
 \[
 \Omega_d^N(Y;\xi)\to  \Omega_d^N(X,\xi) \quad \text{ and } \quad \pi_d(Y^\xi)  \to \pi_d(X^\xi)
 \]
 are isomorphisms. Consequently, we may additionally assume that $X$ is a finite complex.
 
Then any element $\alpha \in \pi_d(X^\xi)$ yields a normal structure over $X$, namely, $(X,\alpha)$.  It follows that
$\phi$ is onto. Similarly, if $(Q,\alpha)$ is a normal structure over $X$ such that $\phi(Q,\alpha) = 0$, then the composition
\[
S^d \to Q^\xi \to X^\xi
\]
is the trivial homotopy class; represent it by a map. Let $W$ be the mapping cylinder of $Q\to X$. Then one may choose a null homotopy
$(D^{d+1},S^d) \to (W^\xi,Q^\xi)$. We infer that  the  bordism class of $(Q,\alpha)$ is trivial, so $\phi$ is one-to-one.
\end{proof}

\subsection{Poincar\'e Bordism}
Let $P_0$ and $P_1$ be Poincar\'e spaces over $X$ of dimension $d$ with normal invariants $\alpha_0$ and $\alpha_1$. 
Then $P_0$ and $P_1$ are {\it bordant} if 
 \[
 P_0 \amalg P_1
 \]
 bounds a Poincar\'e  pair over $X$ of dimension $d+1$ with normal invariant  \eqref{eqn:pinch}.
 
The associated abelian group of equivalence classes
 is denoted by 
 \[
 \Omega_d^P(X,\xi)\, .
 \]
 
\subsection{Bordism of (normal, Poincar\'e) pairs}
Similarly, we let
\[
\Omega_{d}^{\NP}(X,\xi)
\] 
denote the bordism group of (normal, Poincar\'e) pairs over $X$ of dimension $d$.  
Then for formal reasons,  there is a long exact sequence 
\[
\cdots \to \Omega_{d+1}^{\NP}(X,\xi) @> \partial >>  \Omega_d^P(X,\xi) \to \Omega_d^N(X,\xi)@ >>> 
\Omega_{d}^{\NP}(X,\xi)@> \partial >> \cdots\, .
\]
\begin{defn}\label{defn:Fd}
Let $\cal F_d(X,\xi)$ denote the bordism group of triples
\[
 (P,\xi_P,\alpha)
\]
in which
 \begin{enumerate}[(i).]
 \item  $P\in \Top_{/X}$ is homotopy finite, 
 \item $\xi_P$ is the pullback of $\xi$ to $P$, 
\item and 
\[
\alpha\in \pi_{d+1}(X^\xi,P^\xi)
\]
is such that the triple $(P,\xi_P,\alpha_{|S^d})$ is a Poincar\'e space over $X$ of dimension $d$.
\end{enumerate}
 (In the last display we have slightly abused notation in that $X^\xi$ should be replaced by the mapping
cylinder of the map $P^\xi \to X^\xi$.) 
\end{defn}

\begin{prop}  \label{prop:phi} There is a natural isomorphism
\[
\Omega_{d+1}^{\NP}(X,\xi) \cong \cal F_d(X,\xi)\, .
\]
Consequently, there is a long exact sequence of abelian groups
\[
\cal F_d(X,\xi) \to  \Omega_d^P(X,\xi) \to \pi_d(X^\xi)  @> \partial >> \cal F_{d-1}(X,\xi)\to \cdots
\]
\end{prop}

\begin{proof}  Since we have  an isomorphism $\Omega^N_d(X,\xi) \cong \pi_d(X^\xi)$, we only need to
 identify $\Omega_{d+1}^{\NP}(X,\xi)$ with $\cal F_d(X,\xi)$.
But this is straightforward: A homomorphism $\Omega_{d+1}^{\NP}(X,\xi) \to \cal F_d(X,\xi)$ is defined by forgetting data. It is an isomorphism
by an argument similar to the one that proves Lemma \ref{lem:phi}.
\end{proof}



\section{The surgery obstruction of a normal space}\label{sec:Quinn}

The goal of this section is to outline
Ranicki's definition of the surgery obstruction of a normal space \cite{Ranicki}, \cite{Ranicki_top}.

\subsection{Quadratic L-theory} 
The quadratic $L$-theory in degree $d$ of a ring with involution $R$ is  
 the abelian group $L_d(R)$
 of bordism classes of quadratic algebraic Poincar\'e complexes over $R$ of dimension $d$
 (see \cite[prop.~3.2]{Ranicki-1}).

 A chain complex $C$ of projective (left) $R$-modules is said to be {\it homotopy finite} if it is
 chain homotopy equivalent to a finite chain complex, i.e., a bounded chain complex
of finitely generated free $R$-modules which is supported in non-negative degrees.
 
  A {\it quadratic algebraic Poincar\'e complex} of dimension $d$ over a ring $R$
  consists of a homotopy finite $R$-chain complex $C$ equipped with a $d$-cycle
 \[
 \psi \in Z_d((C\otimes_R C)_{h\Bbb Z_2})\, ,
 \]
 called a {\it quadratic structure;} the latter is required to satisfy a certain constraint. Here,  we have used the involution to define
 the tensor product $C\otimes_R C$ and 
   $(C\otimes_R C)_{h\Bbb Z_2}$ is the algebraic quadratic construction, i.e.,
  the homotopy orbits of $\Bbb Z_2$ acting on $C\otimes_{R} C$ (cf.~\ref{rem:homotopy-orbits} below). 
  Let  $\tr(\psi) \in Z_d(C\otimes_{R} C)$ be the image of $\psi$ with respect to the transfer map
   $\tr\: (C\otimes_{R} C)_{h\Bbb Z_2} \to C\otimes_{R} C$. Then the constraint on $\psi$ is
  that the slant product
  \[
  \tr(\psi)/\: C^{d-\ast} \to C_\ast
  \] 
  should be a quasi-isomorphism, where $C^{k} = \hom_R(C_{k},R)$.
 In particular, $C$ satisfies Poincar\'e duality: $H^{d-r}(C) \cong H_r(C)$.
  We refer the reader to \cite[\S3]{Ranicki-1} for the details.
   
 \subsection{The surgery obstruction of a normal space} 
 Given a spherical fibration $\xi \: X\to BG$, with $X$ connected, 
we set $\pi = \pi_1(X)$, and $w = w_1(\xi)$.
Let $(Q,\alpha)$ be a normal space over $X$.
Let 
\[
\Delta \: Q^\xi  \to Q^\xi \smsh (Q_+)
\] 
be the reduced diagonal.
Then the composition
\[
S^d @>\alpha >> Q^\xi @> \Delta >> Q^\xi \smsh (Q_+) 
\]
is an element of the abelian group $\pi_d(Q^\xi \smsh (Q_+) )$. By $S$-duality,
$\pi_d(Q^\xi \smsh (Q_+) )$ is canonically isomorphic to the  group of homotopy classes of maps
of spectra 
\[
\Sigma^d D(Q^\xi) \to  \Sigma^\infty (Q_+)\, ,
\] where $D(Q^\xi)$ is the $S$-dual of $Q^\xi$, i.e., the function spectrum
$F(Q^\xi,S^0)$.  

By pulling back along the $\pi$-covering space  $\tilde Q \to Q$ and applying
Ranicki's $\pi$-equivariant $S$-duality (\cite[\S3]{Ranicki}), we obtain in a similar way a $\pi$-equivariant stable map
\[
\Sigma^d D(\tilde Q^\xi) \to \Sigma^\infty(\tilde Q_+)\, .
\]
The latter in turn gives rise to a map of chain complexes of free $\Bbb Z[\pi]$-modules
\[
\cap [Q]\: C^{d-*}(\tilde Q; L) \to C_{*}(\tilde Q_+)\, ,
\]
where in the latter, we have taken cellular chains and
$[Q]$ is a $d$-cycle in $C_d(\tilde Q;L) = C_d(Q^\xi)$ representing the fundamental class as defined 
by the Hurewicz image of $\alpha$.

More generally, suppose that $(Q,\partial Q)$ is a (normal, Poincar\'e) pair of dimension $d$ over $X$.
Then we obtain a chain map
\[
\cap [Q]\: C^{d-\ast}(\tilde Q,\partial \tilde Q; L) \to C_{*}(\tilde Q_+)\, .
\]
Let $C_\ast$ denote the homotopy fiber of the latter, i.e., the desuspension of the algebraic mapping cone
of $\cap [Q]$.

Ranicki equips $C_\ast$  with the structure of a quadratic algebraic Poincar\'e complex of dimension $d-1$
over $\Bbb Z[\pi]$ with $w$-twisted involution \cite[p.~601]{Ranicki_exact}, \cite[p.~45]{Ranicki_top}.
Hence, there is a  $(d-1)$-cycle  
\[
\psi \in Z_{d-1}((C_*\otimes_{\Bbb Z[\pi]} C_*)_{h\Bbb Z_2})
\]
defined up to contractible choice, whose associated slant product map
\[
C^{d-1-*}\to C_\ast
\]
 is a quasi-isomorphism. 
 
 \begin{rem} Ranicki assumes that $Q$ is a finite CW complex. 
Since we are assuming that $Q$ is merely homotopy finite and cofibrant, the cellular
chain complex $C_\ast = C_\ast(\tilde Q)$ is homotopy finite over $\Bbb Z[\pi]$.
A choice of chain equivalence from $C_\ast$ to a finite $\Bbb Z[\pi]$-chain complex
identifies our construction with that of Ranicki up to equivalence.
\end{rem}

 \begin{rem} \label{rem:homotopy-orbits} The chain complex $(C_*\otimes_{\Bbb Z[\pi]} C_*)_{h\Bbb Z_2}$ is defined as
\[
W_{\%} C := W_\ast \otimes_{\Bbb Z[\Bbb Z_2]} (C_*\otimes_{\Bbb Z[\pi]} C_*)\, ,
 \]
 where $W_\ast$ is the acyclic chain complex defined by the 
minimal $\Bbb Z[\Bbb Z_2]$-resolution of $\Bbb Z$, and $\Bbb Z_2$ acts on $C_*\otimes_{\Bbb Z[\pi]} C_*$
by permuting factors.
  
 When defining the tensor product $C_*\otimes_{\Bbb Z[\pi]} C_*$,
 one uses the $w$-twisted involution of $\Bbb Z[\pi]$ defined by
 \[
 g \mapsto w(g)g^{-1}\, , \quad g\in \pi
 \]
 to equip the left $\Bbb Z[\pi]$-chain complex $C_\ast$ with the structure of a right $\Bbb Z[\pi]$-chain complex. Alternatively, one may define
 $C_*\otimes_{\Bbb Z[\pi]} C_*$ as 
 \[
 \Bbb Z^w \otimes_{\Bbb Z[\pi]} (C_*\otimes_{\Bbb Z} C_*)\, ,
 \]
 where $\Bbb Z^w$ is the $\Bbb Z[\pi]$-module defined by $w_1(\xi)$ and $C_*\otimes_{\Bbb Z} C_*$ is  considered
 as a left $\Bbb Z[\pi]$-chain complex using the standard diagonal action of $\pi$.
 \end{rem}
 
\begin{notation} When $R = \Bbb Z[\pi]$ is an integral group ring and $w\: \pi \to \{\pm 1\}$ is a homomorphism, 
it is customary to denote the quadratic $L$-group $L_d(R)$  by 
  \[
  L_d(\pi,w)\, .
  \]
  \end{notation}
  
  \begin{rem}  Ranicki's quadratic $L$-groups are canonically isomorphic to the corresponding Wall $L$-groups \cite{Wall}, 
\cite[p.~93, prop.~5.3, and \S4-5]{Ranicki-1} .
\end{rem}
  
The quadratic algebraic Poincar\'e bordism class of
$(C_\ast,\psi)$ defines  an element of the quadratic $L$-group
\[
\sigma'(Q,\partial Q) \in L_{d-1}(\pi,w)\, .
\]
This is the {\it surgery obstruction} of the (normal, Poincar\'e) pair $(Q,\partial Q)$.

 If the bordism class  of the (normal, Poincar\'e) pair $(Q,\partial Q)$  admits a  Poincar\'e pair representative,
then $\sigma'(Q,\partial Q) = 0$.

The surgery obstruction  defines a homomorphism
\[
\sigma' \: \cal F_d(X,\xi) \to L_d(\pi,w)\, .
\]

\section{Poincar\'e normal maps}\label{sec:normal-maps}
 Let $(P^d,\xi_P,\alpha_P)$ and $(Q^d,\xi,\alpha)$ be Spivak normal structures. 
In particular, $P$ and $Q$ are equipped with fundamental classes.

\begin{defn} A {\it normal map} from $P$ to $Q$ consists of a pair 
\[
(f,b)
\] in which 
\begin{enumerate}[(i).]
\item $f\: P^d \to Q^d$ is a map of degree one (i.e., $f_\ast([P]) = [Q] \in H_d(Q;L)$), and 
\item $b\: \xi_P \simeq f^\ast \xi$ is a fiber homotopy equivalence (i.e.,
$b\: P \times I \to BG$ is a homotopy from $\xi_P$ to $\xi \circ f$).
\end{enumerate}
Similarly, one has the notion of a normal map of Poincar\'e pairs.

A {\it normal bordism} from $(f_0,b_0) \: P_0 \to Q$ to $(f_1,b_1)\: P_1 \to Q$ is
a normal map of Poincar\'e triads 
\[
((f,b);(f_0,b_0),(f_1,b_1)) \: (W;P_0,P_1) \to (Q\times I;Q\times\{0\},Q\times\{1\})\, .
\]
A normal bordism is sometimes called a  {\it normal cobordism}.
Let 
\[
\cal N(Q,\xi)
\] denote the set of cobordism classes of normal maps with target $Q$.
\end{defn}

\begin{rem} The fiber homotopy equivalence data creates notational clutter.  We will see that there are three other notions of Poincar\'e normal map. All four notions
become equivalent upon passing to normal cobordism, provided that the map $\xi\: Q\to BG$ has been converted to a fibration. Two of the
 three notions do not require specifying a fiber homotopy equivalence. This is a feature unique to the Poincar\'e duality category.
 \end{rem}

\begin{defn}  A {\it strict normal map}
\[
f\: P^d \to Q^d
\]
is a map of spaces 
such that
\begin{enumerate}[(i).]
\item  $f$ is a map over $BG$, i.e., $\xi \circ f = \xi_P$, and
\item $f$ is of degree one:  $f_\ast([P]) = [Q]$.
\end{enumerate}
\end{defn}

Similarly, the notion of strict normal map extends to Poincar\'e pairs and cobordisms. Let
$\cal N'(Q,\xi)$ denote the cobordism classes of strict normal maps with target $Q$.

\begin{lem} \label{lem:strict-normal} Assume that the map $\xi\: Q\to BG$ is a fibration. Then the forgetful function $\cal N'(Q,\xi) \to \cal N(Q;\xi)$
is a bijection.
\end{lem}

\begin{rem} The map $\xi\: Q\to BG$ factors functorially as an acyclic cofibration followed by a fibration
\[
Q \overset{{}_\sim}\rightarrowtail Q' \overset{\xi'}\twoheadrightarrow BG\, .
\]
For this reason, there is no loss in generality in assuming at the outset that the map $\xi\: Q\to BG$ is a fibration.
\end{rem}

\begin{proof}[Proof of Lemma \ref{lem:strict-normal}]
 Let  $(f,b)\: P \to Q$ represent an element of  $\cal N(Q, \xi)$.
Choose a homotopy $H_t\: P \to BG$ from $H_0 = \xi_P$ to $H_1 = f^\ast \xi = \xi \circ f$. Use the homotopy
extension property to lift $H_t$ to a homotopy $\tilde H_t\: P \to Q$ such that $\tilde H_0 = f$. Set $f' = H_1$.
Let $b' \: \xi_P \simeq f^\ast\xi$ be given by the constant homotopy $P \times I \to BG$ from $\xi_P$ to $(f')^\ast\xi = \xi_P$.
Then $f'$ is a strict normal map and $(f,b)$ and $(f',b')$ are normally cobordant. Consequently, the forgetful map is onto.
A relative version of the above shows that the forgetful map is also one-to-one. We leave the details of the latter
assertion to the reader.
\end{proof}

In the definition of normal map given above, it was not required that $f$ preserve the normal invariants. 
The following definition makes this modification. In the following, we
 do not need to assume  that $\xi\: Q\to BG$ is a fibration.

\begin{defn} [Compare {\cite[p.~73]{HV}}]\
 A {\it morphism} of Spivak normal structures 
 \[
 (f,b)\: (P^d,\xi_P,\alpha_P) \to (Q^d,\xi ,\alpha)
 \]
consists of 
\begin{enumerate}[(i).]
\item a map $f\: P \to Q$, and 
\item a fiber homotopy equivalence $b\: \xi_P \simeq f^\ast\xi$, such that
\[
f_\ast(\alpha_P) = \alpha\, ,
\]
where $f_\ast\: \pi_d(P^{\xi_P}) \cong \pi_d(P^{f^\ast\xi}) \to  \pi_d(Q^{\xi})$ is the homomorphism
induced by $(f,b)$.
\end{enumerate}

The set of cobordism classes of morphisms of Spivak normal structures with target $(Q,\xi ,\alpha)$ is denoted by
\[
\cal N(Q,\xi,\alpha)\, .
\]
\end{defn}

There is an obvious map
 
\begin{equation} \label{eqn:phi}
\cal N(Q,\xi,\alpha) \to \cal N(Q,\xi)\, . 
\end{equation}

\begin{lem} \label{lem:automorphism} If $(f,b) \: P \to Q$ is a normal map, then there
is a fiber homotopy equivalence $b' \: \xi_P \to f^\ast \xi$ such that
$(f,b')$ is a morphism of Spivak normal structures. Furthermore,
$(f,b)$ is normally cobordant to $(f,b')$. 

Similarly, if $(f,b)\: (P,\partial P) \to (Q,\partial Q)$ is
a normal map whose restriction  $(f,b)_{|\partial Q}\: \partial P \to \partial Q$ is a map of Spivak normal structures, then
there is a fiber homotopy equivalence $b'\: \xi_P \simeq f^\ast \xi$ which coincides with $b$ on $\partial P$
such that $(f,b')$ is a map of Spivak normal structures. Furthermore, there 
is a normal cobordism rel $\partial P$ from 
$(f,b)$ to $(f,b')$.
\end{lem}

\begin{proof} We only prove the first part, as the proof of the second part follows along similar lines.
Let $H(\xi)$ be the group of homotopy classes of orientation preserving stable fiber homotopy self-equivalences
of $\xi$.
By \cite[thm.~I.4.19]{Browder}, there is 
 a unique element of $\phi\in H(\xi)$ with the property that
 \[
 \phi_\ast(f_\ast (\alpha_P))  = \alpha \, ,
 \]
where $\phi_\ast \: \pi_d(Q^\xi) \to \pi_d(Q^\xi)$ is induced by $\phi$.
Consider the fiber homotopy equivalence $b'\: \xi_P \simeq f^\ast \xi$
defined by the composition
\[
\xi_P @> b >{}^\simeq > f^\ast \xi @> f^\ast \phi > {}^\simeq > f^\ast \xi \, .
\]
Then by construction
 $(f,b')$ is a morphism of Spivak normal structures.
 
 To show that $(f,b)$ is normally cobordant to $(f,b')$, we view $b$ as a homotopy $P \times I \to BG$ from $\xi_P$ to $f^\ast \xi$
  and $\phi$ as a homotopy $Q\times I \to BG$ from $\xi$ to itself.
 Then $b'$ is the result of the homotopy $B\: P \times I \to BG$ defined by
 \[
B(x,t)= \begin{cases} 
 b(x,2t) \quad & 0 \le t \le 1/2\, ,\\
 \phi(f(x),2t-1) & 1/2 \le t \le 1\, .
 \end{cases}
 \]
 Let $F = f\times \text{id}\: P \times I \to Q \times I$. Then 
 $(F,B)$ is the desired normal cobordism.
 \end{proof}
 
 \begin{cor} \label{cor:automorphism} The map \eqref{eqn:phi} is a bijection.
\end{cor}

The last notion of normal map combines the preservation of normal invariants with 
the strictness constraint:

\begin{defn}  A {\it strict morphism} of Spivak normal structures 
\[
f\: (P,\xi_P,\alpha_P) \to (Q,\xi ,\alpha)
\]
is a map of spaces $f\: P \to Q$
such that
\begin{enumerate}[(i).]
\item  $f$ is a map over $BG$, and
\item $f_\ast(\alpha_P) = \alpha$.
\end{enumerate}

Similarly, one has the notion of strict cobordisms of Spivak normal structures. The set of cobordism classes
of strict morphisms with target $(Q,\xi ,\alpha)$ is denoted by $\cal N'(Q;\xi,\alpha)$.
\end{defn}

\begin{rem}  A strict normal map is also known as a {\it formally $d$-dimensional normal map}
in the sense of \cite[p.~599]{Ranicki_exact}.
\end{rem}

The proof of the following is essentially the same as the proof of Lemma \ref{lem:strict-normal}; the details are  left to the reader.

\begin{lem} \label{lem:strict-Spivak} Assume $\xi\: Q \to BG$ is a fibration. Then the forgetful map $\cal N'(Q;\xi,\alpha)  \to \cal N(Q;\xi,\alpha)$ is a bijection.
\end{lem}

\begin{cor}\label{cor:split-normal} If the map $\xi\: Q\to BG$ is a fibration,  then every  map of the commutative square 
\[
\xymatrix{
\cal N'(Q,\xi,\alpha) \ar[r]\ar[d] & \cal N(Q,\xi,\alpha) \ar[d] \\
\cal N'(Q,\xi) \ar[r] & \cal N(Q,\xi) 
}
\]
is a bijection.
\end{cor}

\begin{rem} \label{rem:norma-map-convention} Unless otherwise specified, the map $\xi\: Q\to BG$ is assumed to be a fibration and
a normal map refers to a representative of the set $\cal N'(Q,\xi)$, i.e., a strict normal map $f\: P \to Q$. The same convention will
apply to normal maps of pairs and normal cobordisms.
Similarly, when discussing morphisms of Spivak normal structures, we mean them in the strict sense.
\end{rem}

\begin{defn} Suppose that $X$ is a space equipped with a stable spherical fibration $
\xi\: X\to BG$. Let $Q^d$ be a Poincar\'e space over $X$ of dimension $d$.
By a {\it Poincar\'e normal map over $X$} we mean a morphism $f\: P \to Q$ of $\Top_{/X}$ which is
a Poincar\'e normal map when considered in $\Top$.
\end{defn}

\section{The surgery obstruction of a normal  map and Theorem \ref{bigthm:sigma-onto}} \label{sec:bordism}

\subsection{The surgery obstruction of a Poincar\'e normal map}  Let $f\: M^d \to X^d$ be a normal map where $X$ is connected.
By Corollary \ref{cor:split-normal}, we assume $f$ is a strict morphism $(M,\xi_M,\alpha_M) \to (X,\xi,\alpha)$ 
of Spivak normal structures (cf.~ Remark \ref{rem:norma-map-convention}). Set $\pi = \pi_1(X)$ and $w = w_1(\xi_X)$.

By Ranicki's algebraic theory of surgery \cite{Ranicki}, $f$ has a surgery obstruction 
\[
\sigma(f) \in L_d(\pi,w)\, .
\]
Let $\tilde X \to X$ be a universal cover and let 
$\tilde f\: \tilde M \to \tilde X$ be the base change of $f$ to $\tilde X$. Then $\tilde f$ is $\pi$-equivariant.
The stable spherical fibration $\xi$ over $X$ pulls back as well to $\tilde M$ so
we can take the induced map of Thom spectra with (naive) $\pi$-action
\[
\tilde f^\xi\: \tilde M^{\xi_M} \to \tilde X^\xi\, .
\]
Since $X$ has the homotopy type of a finite complex, both 
$\tilde M^{\xi_M}$ and $\tilde X^\xi$ are homotopy finite objects in the category of spectra with $\pi$-action.
In particular, we can take the equivariant $S$-dual in the sense of \cite[\S3]{Ranicki} which results in an equivariant stable {\it umkehr map}
\[
f^!\: \tilde X_+  \,\, \to  \, \tilde M_+\, .
\]
Let $C_*(f^!)$ denote the algebraic mapping cone
of  the chain map associated with $f^!$.  Ranicki shows that  $C_*(f^!)$ is equipped with the
 structure of a $d$-dimensional quadratic algebraic Poincar\'e complex over $\Bbb Z[\pi]$ \cite[prop.\ 2.3]{Ranicki}. 
 As previously remarked, the bordism group of such data coincides with Wall's $L$-group $L_d(\pi,w)$, where $w = w_1(\xi)$ 
and the algebraic bordism class of $C_*(f^!)$ is the surgery obstruction $\sigma(f)$ of $f$ \cite[\S7]{Ranicki}.

A normal bordism $(F,f_0,f_1)$  induces in an evident way an algebraic cobordism from 
$C_*(f_0^!)$ to $C_*(f_1^!)$  (see \cite[prop.~6.6]{Ranicki}).
Therefore $\sigma(f_0) = \sigma(f_1)$ and we conclude:

\begin{lem}[Cobordism Invariance] \label{lem:Ranicki-surgery-obstruction} 
 The surgery obstruction $\sigma$ is a normal bordism invariant and defines a function
\begin{equation} \label{eqn:Ranicki-surgery-obstruction}
\sigma\: \cal N(X,\xi) \to L_d(\pi,w) \, .
\end{equation}
In particular, if $f$ is normally cobordant to a homotopy equivalence, then $\sigma(f) = 0$.
\end{lem}

More generally, if $f\:(M,\partial M) \to (X,\partial X)$ is a normal map of Poincar\'e pairs  whose restriction $\partial M \to \partial X$ is a homotopy equivalence, then
the surgery obstruction $\sigma(f) \in L_d(\pi,w)$ is defined similarly. It vanishes if $f$ is normally cobordant rel $\partial M$ to a homotopy equivalence.





\begin{proof}[Proof of Theorem \ref{bigthm:sigma-onto}]  Let $j\: K \to X$ be a 2-skeleton, i.e., a $2$-connected map from a finite complex of dimension $\le 2$.
Then by \cite[thm.~A]{Klein_haef} (cf.~Theorem \ref{my_thm}), there is a Poincar\'e embedding of $j$, i.e, there is a
homotopy pushout square
\[
\xymatrix{
\partial \bar K \ar[r] \ar[d] & X_0\ar[d] \\
\bar K \ar[r] & X
}
\]
and a homotopy equivalence $K @> \simeq >> \bar K$ such that the composition $K \to \bar K\to X$ is $j$
(cf.~\S\ref{subsec:embedding}).
Here, $(\bar K,\partial \bar K)$ and $(X_0,\partial \bar K)$ are Poincar\'e pairs of dimension $d$ whose fundamental classes
 are compatible with the fundamental class of $X$. Furthermore, the inclusion $\partial \bar K\to \bar K$ is $(d-3)$-connected.
As a consequence,  $X$ is identified with the double mapping cylinder
\[
\bar K\cup_{\partial \bar K \times \{0\}} (\partial \bar K \times I)\cup_{\partial \bar K\times \{1\}} X_0\, .
\]
 By Wall realization in the Poincar\'e category (Theorem \ref{thm:Wall_realization}),
if $\tau \in L_d(\pi,w) $ is an element, there is a Poincar\'e pair
$(W,\partial W)$  of dimension $d$ with $\partial W = \partial_0 W\amalg \partial_1 W$,
and a normal map 
\[
F\: (W,\partial_0 W,\partial_1 W)  \to (\partial \bar K \times [0,1],\partial \bar K\times \{0\},\partial \bar K\times \{1\})
\]
such that $\partial_0 W \to \partial \bar K\times \{0\}$ is the identity,
$\partial_1 W \to \partial \bar K \times \{1\}
$ is a homotopy equivalence, and the surgery obstruction
of $F$ rel boundary is  $\tau$.
\begin{figure} 
\centering
\begin{tikzpicture}[scale=.75]
    \draw[thick] (0,0) ellipse (3cm and 2cm);
    
    \draw[thick, black] plot[smooth cycle, tension=1,scale=1.2] coordinates {(0,0.5) (1,1) (1.5,0) (1,-1) (0,-0.5) (-1,-1) (-1.5,0) (-1,1)};
       \fill[white!60!gray!60] plot[smooth cycle, tension=1,scale=1.2] coordinates {(0,0.5) (1,1) (1.5,0) (1,-1) (0,-0.5) (-1,-1) (-1.5,0) (-1,1)};
    \draw[thick, scale=.2] (0,0) ellipse (3cm and 2cm);
    \fill[white, scale=.2](0,0) ellipse (3cm and 2cm);
    \draw[thick,->] (3.5,0) to[bend left=20] (6.5,0);
    \node at (5,.6) {{\tiny $f$}};
     \node at (0,0) {{\tiny $\bar K$}}; 
      \node at (1,.5) {{\tiny $W$}};
         \node at (2.5,0) {{\tiny $X_0$}};
\end{tikzpicture} \hskip .2cm
\begin{tikzpicture}[scale=.75]
    \draw[thick] (0,0) ellipse (3cm and 2cm);
         \draw[thick, scale =.6] (0,0) ellipse (3cm and 2cm);
           \fill[white!60!gray!60, scale =.6] (0,0) ellipse (3cm and 2cm);
        \draw[thick, scale =.2] (0,0) ellipse (3cm and 2cm);
         \fill[white,scale =.2](0,0) ellipse (3cm and 2cm);
     \node at (0,0) {{\tiny $\bar K$}}; 
      \node at (.78,.65) {{\tiny $\partial \bar K {\times} I$}};
         \node at (2.5,0) {{\tiny $X_0$}};
\end{tikzpicture}
\caption{A depiction of the Poincar\'e normal map  constructed in the proof of Theorem \ref{bigthm:sigma-onto}. 
The white region $\bar K \amalg X_0$ is mapped using the identity. On the gray region,  the map
supplied by Theorem \ref{thm:Wall_realization} is used.
}
\label{fig:schematicF}
\end{figure}
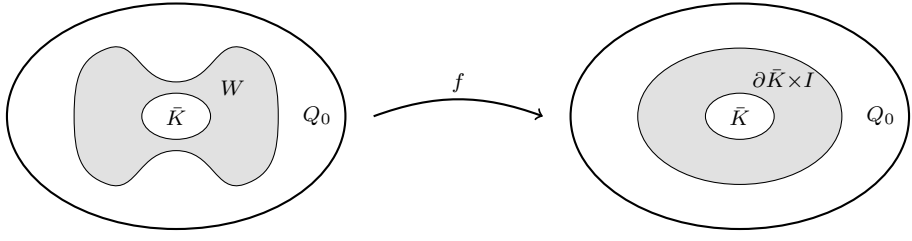

Let $f\: M \to X$ be defined as 
\[
 M := \bar K\cup_{F_{|\partial_0 W}} W \cup_{F_{|\partial_1 W}} X_0  @> \text{id}_{\bar K} \cup F \cup \text{id}_{X_0} >>
\bar K\cup_{\partial \bar K \times \{0\}} (\partial \bar K \times I)\cup_{\partial \bar K\times \{1\}} X_0 = X\, .
 \]
Then $f$  is a Poincar\'e normal map with surgery obstruction $\tau$  (cf.~\!fig.~\!\ref{fig:schematicF}).
\end{proof}

\subsection{Normal maps and normal pairs} Let $X$ be a space of finite type.
Let $f\: P \to Q$ be a Poincar\'e normal map over $X$
of dimension $d$. Assume that $\pi_1(Q) \to \pi_1(X) = \pi$ is an isomorphism. 
Let $M_f$ denote the mapping cylinder
of $f$ and set $\partial M_f := P \amalg Q$. Then $\partial M_f$ is a Poincar\'e duality space
of dimension $d$.
Moreover, if $f$ is a strict morphism of Spivak normal structures, then  $(M_f,\partial M_f)$  is equipped with the structure of
 a (normal, Poincar\'e) pair over $X$ of dimension $d+1$.

\begin{lem} \label{lem:normal=quinn} The surgery obstruction of $f$ coincides with 
the surgery obstruction of the (normal, Poincar\'e) pair  $(M_f,\partial M_f)$. That is,
\[
\sigma(f)  = \sigma'(M_f,\partial M_f) \in L_d(\pi,w)\, .
\]
\end{lem}

\begin{proof}  The surgery obstruction of $f$ is the quadratic algebraic
Poincar\'e complex of dimension $d$ defined by the chains on the homotopy
cofiber of the stable $\pi$-equivariant umkehr map
\[
f^!\: \tilde Q_+ \to \tilde P_+\, .
\]
There is a weak equivalence of $\pi$-spaces $\tilde M_f/\partial \tilde M_f \simeq \Sigma (\tilde P_+)$. With respect to this
identification, an
unraveling of the construction shows that the surgery obstruction of the normal pair $(M_f,\partial M_f)$ is given by
the one-fold desuspension of the chains on the cofiber of the map
\[
f_+\: \tilde P_+\to \tilde Q_+\, .
\]
The result then follows immediately from the fact that $f_+ \circ f^!$ is stably equivariantly homotopic to the identity map of $Q_+$ \cite[prop.~2.2]{Ranicki}.
\end{proof}



\subsection{The function $t$} 
As outlined above, a strict morphism of Spivak normal structures $f\: P \to Q$ over $X$  determines 
a (normal, Poincar\'e) pair $(M_f,P \amalg Q)$ of dimension $(d+1)$ over $X$.
Furthermore, it  determines in an evident way an element of $\cal F_d(X,\xi) $ which is invariant under bordism.

By Corollary \ref{cor:split-normal}, the above recipe defines a function
\begin{equation} \label{eqn:t}
t\: \cal N(Q,\xi) @>>> \cal F_d(X,\xi) \, .
\end{equation}
By Lemma \ref{lem:normal=quinn}, we infer that
the composition
\[
\cal N(Q,\xi) @> t >> \cal F_d(X,\xi) @> \sigma' >> L_d(\pi,w) 
\]
coincides with $\sigma$.

\begin{cor} Assume $t(f) = 0$, $\pi$ is trivial and Theorem \ref{bigthm:fund-thm} holds. 
If $d \ge 7$ or $d=5$, then  $f$ is normally cobordant to a homotopy equivalence.
\end{cor}

\begin{proof} We have $\sigma(f) = \sigma'(t(f)) =0$. Hence, 
$f$ is  normally cobordant to a homotopy equivalence by 
Theorem \ref{bigthm:fund-thm}.
\end{proof}

\begin{rem}
 If $\pi$ is non-trivial, then $f$ is normally cobordant to a  homotopy equivalence provided
that Theorem \ref{bigthm:fund-thm} holds for non-trivial $\pi$.
\end{rem}
 
\subsection{Lifting criteria}  We will show that the function
 $t$ of \eqref{eqn:t} is onto when
 $Q$  is a Poincar\'e space over $X$, provided that certain conditions are met.
 
 To avoid clutter,  we identify 
a map $\phi\: A\to B$  with  the associated pair $(M_\phi,A)$ in which $M_\phi$ is 
the mapping cylinder of $\phi$,  and by slight abuse of notation, we write $(B,A)$ in place of $(M_\phi,A)$.
In particular, a commutative diagram
\[
\xymatrix{
A \ar[d] \ar[r] & C \ar[d] \\
B \ar[r] & D 
}
\]
will be written as a map of ``pairs'' $(B,A) \to (D,C)$.

 Let $(Y,\partial Y)$ be a Poincar\'e pair over $X$ of dimension $d+1$
 with normal invariant $\beta\: (D^{d+1},S^d) \to (Y^\xi,(\partial Y)^\xi)$.
Suppose that
\[
f\: P \to \partial Y
\] is a normal map of Poincar\'e spaces over $X$. 

Let $\alpha_0 \: S^d \to P^\xi$ be the normal invariant of $P$. Then 
the restriction of $\beta$ to $S^d$ coincides with $f\circ \alpha_0$. Let $\alpha$
denote the homotopy class
\[
(D^{d+1},S^d) @>(\beta,\alpha_0)>> (Y^\xi,P^\xi) @>>> (X^\xi,P^\xi)\, .
\]
Then $(P,\alpha)$ represents an element of $\cal F_d(X,\xi)$.

Moreover,
\begin{equation} \label{eqn:norm-pd}
(\partial Y,P \amalg \partial Y) \quad \text{ \rm and } \quad \quad (Y,P)
\end{equation}
are  (normal, Poincar\'e) pairs over $X$.

\begin{prop}\label{prop:bordism-lift}  The (normal, Poincar\'e) pairs 
\eqref{eqn:norm-pd}
 represent the same element
of 
 $\Omega_{d+1}^{\NP}(X,\xi)$. Equivalently, $t(f) \in \cal F_d(X,\xi)$ coincides with 
 the class defined by $(P,\alpha)$.
\end{prop}

\begin{proof} The (normal, Poincar\'e) pairs $(\partial Y,P \amalg \partial Y)$ and $(Y,P \amalg \partial Y)$ 
represent the same element of  $\Omega_{d+1}^{\NP}(X,\xi)$, where the mapping cylinder of the inclusion $\partial Y\to Y$
defines a bordism. Consequently,
for the first part it suffices to show that $(Y,P \amalg \partial Y)$ and $(Y,P)$ are bordant.

Observe that $Y$ defines a Poincar\'e bordism over $X$ from $\partial Y$ to the empty space.
This bordism extends  to the (normal, Poincar\'e) bordism
\[
(Y,(P \times [0,1]) \, \amalg\,  Y)
\]  
from
$
(Y,(P \times \{0\})\,  \amalg \, \partial Y)
$
to $(Y,P \times \{1\})$. 

For the second part, we only need to observe that  $(Y,P \amalg \partial Y)$ corresponds to  $t(f)$ and $(Y,P)$ corresponds to $(P,\alpha)$ with respect to the isomorphism
$\Omega_{d+1}^{\NP}(X,\xi) \cong \cal F_d(X,\xi)$.
\end{proof}

\begin{prop} \label{better-reps} Assume $d \ge 7$ or $d=5$.
Then every element of $\cal F_d(X,\xi) $
admits a representative $(P,\alpha)$,  in which the map
$P \to X$ is
$\lfloor d/2\rfloor$-connected.
\end{prop}

\begin{proof} This is a direct consequence of Theorem \ref{below_middle} below, which says that surgery below the middle dimension
can be performed.
\end{proof}

We assume for the rest of this section that $X$ has the homotopy type of a finite CW complex of dimension $\le s$.
Suppose that $(P,\alpha)$ is as in Proposition \ref{better-reps}. If $d \ge 2s+1$, then the structure map $P\to X$
is $s$-connected, so it admits
a section up to homotopy $v\: X\to P$.

\begin{lem} \label{P-decomp} Assume $d\ge \max(5,2s+1)$.
 Then  $v\: X\to P$ underlies a Poincar\'e embedding (cf.~\ref{subsec:embedding}). In particular,  there is a  homotopy pushout square
 \[
\xymatrix{
\partial \bar X \ar[r] \ar[d] & P_0\ar[d]\\
 \bar X \ar[r]_v & P 
 }
 \]
 in which  $(\bar X,\partial \bar X)$ is a $(d-s-1)$-connected Poincar\'e pair of dimension $d$,
and there is a factorization of $v$
\[
X@>>> \bar X \to P
\]
in which $X\to \bar X$ is a homotopy equivalence.
 \end{lem}

\begin{proof} The Poincar\'e embedding of $v$ exists by Theorem \ref{my_thm} below.
\end{proof}

\begin{cor} \label{cor:P-decomp}  Assume that $d \ge  \max(5,2s+1)$ with $(P,\alpha)$ as above.
Then there is a Poincar\'e normal map $f\: P \to Q$ over $X$  such that the mapping cylinder
of $Q\to X$ defines
 a Poincar\'e pair of dimension $d+1$.
\end{cor}

\begin{proof} 
Let $Q = S_{X} {\partial \bar X}$, where $\partial \bar X$ is as in Lemma \ref{P-decomp}. The mapping
cylinder of  the map
$Q\to X$ defines a Poincar\'e pair over $X$ of dimension $d+1$, since it is identified up to homotopy with the Poincar\'e pair
\[
(\bar X \times I,\partial (\bar X \times I)) := (\bar X \times I, \bar X \times \{0\} \cup  \partial \bar X \times I \cup \bar X \times \{1\})\, .
\]
Let $f\: P \to Q$ be the map over $X$ given by
\[
P = X \cup (\partial\bar  X \times I) \cup P_0 @> \text{id}_{X} \cup \text{id}_{\partial \bar  X \times I}
 \cup g >>  X \cup  ( \partial\bar X \times I) \cup X = Q\, ,
\]
where $g\:P_0\to X$ is the composition $P_0 \to P  \to X$.
Then $f$ is compatible with $\xi$, so it underlies a Poincar\'e normal map.
\end{proof}

\begin{rem} \label{rem:uniqueQ} The map $Q \to X$ is the unique Poincar\'e $(d+1)$-thickening
of $X$ up to concordance with Spivak fibration $\xi$ 
by Corollary \ref{cor:thickening_stable_range}  below.
 \end{rem}

\begin{prop} \label{prop:t-onto} Assume $d \ge \max(5,2s+1)$ and let $Q\to X$ be as in Corollary \ref{cor:P-decomp}.
Then the function  $t\: \cal N(Q,\xi) \to \cal F_d(X,\xi) $ is onto.
\end{prop}

\begin{proof}
Let $(P,\alpha)$  represent an 
element $x\in  \cal F_d(X,\xi) $. Then by the proof of Corollary \ref{cor:P-decomp},  there is a normal map
$f\: P \to Q$ over $X$.
By Proposition \ref{prop:bordism-lift}, we infer that $t(f)=x$. \end{proof}

\begin{cor} \label{cor:t-onto} Assume $d \ge 5$ and that $X$ has the homotopy type of a CW complex of dimension $\le 2$. Then
 $t\: \cal N(Q,\xi) \to \cal F_d(X,\xi) $ is onto.
 \end{cor}



\section{Proof of Theorem \ref{bigthm:stronger-exactness} and Addendum  \ref{bigadd:stronger-exactness}} \label{sec:exact-sequence}

The proof of Theorem \ref{bigthm:stronger-exactness} and Addendum  \ref{bigadd:stronger-exactness}
will rely on Theorems \ref{bigthm:fund-thm} and \ref{bigthm:sigma-onto}.
The first step amounts to a kind of ``$\pi$-$\pi$ theorem'' for the abelian group valued functor on $\Top_{/X}$ defined by
$Y \mapsto \cal F_d(Y;\xi)$.

Fix a $2$-skeleton $Y \to X$, which we can take to be an inclusion.
We will show that the map
\[
\cal F_d(Y;\xi) \to \cal F_d(X,\xi)
\]
is an isomorphism if $d \ge 5$. 

Let \[
\cal F_d(X,Y,\xi)
\]
be the bordism group generated by $(P,\partial P,\alpha)$, in which 
$(P,\partial P)$ is a Poincar\'e pair of dimension $d$ over $(X,Y)$ and $\alpha$ is a homotopy class of a diagram
of the form
\[
\xymatrix{
S^d \ar[r] \ar[d] & P^\xi/(\partial P)^\xi \ar[d] \\
D^{d+1} \ar[r]  & X^\xi/Y^\xi
}
\]
in which the top horizontal arrow is a normal invariant.
Then one has a long exact sequence of abelian groups
\[
\cdots \to \cal F_d(X,Y,\xi) \to \Omega_d^{P}(X,Y,\xi) \to \pi_{d}(X^\xi/Y^\xi) \to \cal F_{d-1}(X,Y,\xi) \to \cdots
\]
in which $\Omega_d^{P}(X,Y,\xi)$ is the bordism group of Poincar\'e pairs over $(X,Y)$.
One also has a long exact sequence
\[
\cdots \to \cal F_{d+1}(X,Y,\xi)@> \partial >>  \cal F_d(Y,\xi) \to  \cal F_d(X,\xi)\to \cal F_d(X,Y,\xi) @> \partial >> \cdots \, .
\]

\begin{prop}\label{prop:pi-pi} If $d\ge 5$, then $\cal F_d(X,Y,\xi)$ is trivial.
\end{prop}

\begin{rem} We note that Proposition \ref{prop:pi-pi} does not require  $X$ to be $1$-connected.
\end{rem}

\begin{proof}[Proof of  Proposition \ref{prop:pi-pi}]
It suffices  to show that the homomorphism 
\[
\cal F_d(Y,Y,\xi) \to \cal F_d(X,Y,\xi)
\]
is surjective, since
$\cal F_d(Y,Y,\xi)$ is trivial.
Let $(P,\partial P,\alpha)$ represent an element of $\cal F_d(X,Y,\xi)$. By Proposition \ref{better-reps} (or Theorem \ref{below_middle} below), we may 
assume that the map $P \to X$ is $2$-connected. Then by elementary obstruction theory, the map $Y\to X$ factors as
\[
Y  \to P \to X
\]
where $Y \to P$ is $1$-connected.
By Corollary \ref{cor:easy} below, the map $Y \to P$ underlies  a Poincar\'e embedding. Hence, there is a Poincar\'e pair $(\bar Y,\partial \bar Y)$ of dimension $d$, 
a homotopy equivalence $Y @> \simeq >>\bar Y$, and a weak equivalence
\[
(\bar Y \cup_{\partial \bar Y} P_0,\partial P) \simeq (P,\partial P)
\]
whose restriction to $\partial P$ is the identity and whose restriction to $Y$ is the given map $Y\to P$ (cf.~Fig.~\!\ref{fig:pd-embed-Y}). Without loss in generality,
we identify $P$ with $\bar Y \cup_{\partial \bar Y} P_0$. In particular,
there is an inclusion $\iota\: P_0\to P$.
Let $c\: P_0 \to I$ be a continuous function satisfying $\partial \bar Y = c^{-1}(0)$ and $\partial P = c^{-1}(1)$. Then the pair
\[
(T,\partial_+ T)  := (P\times I,P_0)
\]
arising from the inclusion $(\iota,c) \:P_0 \to P\times I$  is a cobordism of Poincar\'e pairs
from  $(\bar Y,\partial \bar Y)$ to $(P,\partial P)$ (cf.~Fig.~\!\ref{fig:cobordism}).
\begin{figure}
\centering
\begin{minipage}[b][4cm][s]{.45\textwidth}
\centering
\vfill
\begin{tikzpicture}[scale=.55]
     \fill[white!60!gray!60] (0,0) ellipse (3 and 2);
    \draw (0,0) ellipse (3 and 2);
       \fill[white] (0,0) circle (1);
           \draw (0,0) circle (1);
    \node at (0,0) {{\tiny $\bar Y$}};
    \node at (-.85,.9) {\scalebox{.5}{$\partial \bar Y$}};
       \node at (1.75,0) {\tiny{$P_0$}};
             \node at (2.5,1.6) {\scalebox{.6}{$\partial P$}};
\end{tikzpicture}
\vfill
\caption{\tiny Poincar\'e embedding of $(\bar Y,\partial \bar Y)$ in $(P,\partial P)$.}\label{fig:pd-embed-Y}
\vspace{\baselineskip}
\end{minipage}\qquad
\begin{minipage}[b][5cm][s]{.45\textwidth}
\centering
\vfill
 \vspace*{-2cm}
\begin{tikzpicture}[scale=.75]
    \draw (0,0) --(5,0);
    \draw (0,2) -- (5,2);

    \draw (5,1) ellipse (0.6 and 1);

    \draw[dashed] (0,1) ellipse (0.6 and 1);
      \draw (0,2) arc  (90:270: 0.6 and 1);

    \draw (5,0) arc (270:90:0.6 and 1);
    \draw[dashed] (5,2) arc (90:270: 0.6 and 1);

    \draw[->] (3.2,2.35) -- (5,2.35);
    \draw[->] (1.9,2.35) -- (0,2.35);
    
\node at (0,1) {{\tiny $\bar Y$}};
\node at (5,1) {{\tiny $P$}};
\node at (2.5,2.3) {{\tiny $P_0 $}};
\node at (2.5,1) {{\tiny $P \times I$}};
\node at (-0.85,1.5) {{\tiny $\partial \bar Y$}};
\node at (5.95,1.5) {{\tiny $\partial P$}};
    \end{tikzpicture}
    \vfill
 \vspace*{-3cm}
    \caption{\tiny Schematic of the Poincar\'e cobordism $(T,T_+)$.}   \label{fig:cobordism}
     \vspace*{.5cm}
\end{minipage}
\end{figure}
Note that the structure  map $(\bar Y,\partial \bar Y) \to (X,Y)$ factors
through $(Y,Y)$. The normal invariant $(\bar Y,\partial \bar Y)$
is given by the composition
\[
S^d @>>> P^\xi/(\partial P)^\xi @>>>  \bar Y^\xi/(\partial Y)^\xi\, ,
\]
 where the first map is the normal invariant of $(P,\partial P)$ and 
 the second map is given by collapsing $P_0^\xi$ to a point. Then the Poincar\'e pair
 $(\bar Y,\partial \bar Y)$ over $(Y,Y)$ together with its normal invariant defines a lift of the equivalence class of $(P,\partial P,\alpha)$ to $\cal F_d(Y,Y,\xi)$.
\end{proof}

\begin{cor} \label{cor:pi-pi} Assume $d \ge 5$.
Then $\cal F_d(Y,\xi) \to \cal F_d(X,\xi)$ is an isomorphism.
\end{cor}

\begin{proof}[Proof of Theorem \ref{bigthm:stronger-exactness} and Addendum \ref{bigadd:stronger-exactness}]
Assume $d \ge 7$ or $d = 5$.
The homomorphism $\sigma'\: \cal F_d(X,\xi) \to L_d(\pi,w)$ is  natural.
Let $Y \to X$ be a two skeleton. Then $\pi_1(Y) = \pi_1(X)$,  $\dim Y \le 2$, and the square
\[
\xymatrix{
 \cal F_d(Y,\xi) \ar[r]^{\sigma'}\ar[d]_{\cong} & L_d(\pi,w) \ar[d]^{=} \\
\cal F_d(X,\xi) \ar[r]_{\sigma'}  & L_d(\pi,w) 
}
\]
commutes. The left vertical map is an isomorphism by Corollary \ref{cor:pi-pi}. Consequently, it will be enough to show
that  
\begin{equation} \label{eqn:sigma-prime-Y}
\sigma' \: \cal F_d(Y,\xi) \to  L_d(\pi,w)
\end{equation} is an isomorphism.  
\medskip

 \noindent{\it Proof that \eqref{eqn:sigma-prime-Y} is surjective:} 
By Lemma \ref{lem:normal=quinn}, the composition $\sigma'\circ t\: \cal N(Y,\xi) \to L_d(\pi,w)$ coincides with the surgery obstruction $\sigma$. 
If $\pi$ is trivial, then by Theorem \ref{bigthm:sigma-onto}, it follows that $\sigma$ is onto since for odd $d$ the
$L$-groups $L_d$ are trivial. Hence, $\sigma'$ is onto. 
We observe that the same argument works when $\pi$ is non-trivial and $d$ is even.
 \smallskip
 
We will exhibit two proofs that $\sigma'$ is injective. The first uses classical manifold surgery; it
was already known, does not assume that $\pi$ is trivial, and 
is valid for $d \ge 5$ \cite[\S4]{Levitt_cobordism}, \cite[p.~96]{HV}.
The second proof is manifold-free but assumes either that $\pi$ is trivial or that Theorem 
\ref{bigthm:fund-thm}  holds for non-trivial $\pi$.
\medskip

\noindent{\it First proof that  \eqref{eqn:sigma-prime-Y} is injective:} 
Since $BO \to BG$ is $2$-connected and $\dim Y \le 2$, elementary obstruction theory implies that
 the map $\xi\: Y \to BG$ factors through $BO$, i.e., $\xi$ lifts to a stable vector bundle. So we may as well assume
that $\xi$ is a stable vector bundle.
Then the composition 
\[
\Omega^{\text{diff}}_d(Y,\xi) @> \iota >> \Omega^P_d(Y,\xi) \to \pi_d(Y^\xi)
\]
is an isomorphism by manifold transversality. 
Moreover, $\sigma'$ factorizes as 
\[
\cal F_d(Y,\xi) \to \Omega^P_d(Y,\xi) \to C \to  L_d(Y,\xi)\, ,
\]
where $C$ is the cokernel of $\iota$ and the composition $\cal F_d(Y,\xi) \to \Omega^{P}_d(Y,\xi) \to C$ is an isomorphism.
We must prove that $C \to L_d(\pi,w)$ is injective.

Let $x\in \Omega^{P}_d(Y,\xi)$ be a lift of any element of $C$. Then $x$ is represented by $(Q,\alpha)$ and by transversality applied to $\alpha\: S^d\to Q^\xi$, there is a normal map 
$f\:P \to Q$ in which $P$ is a smooth manifold. Then the classical manifold surgery obstruction $\sigma(f) \in L_d(\pi,w)$ is defined
and is trivial if and only if $f$ is normally cobordant to a homotopy equivalence, i.e., $\sigma(f) = 0$ if and only if $(Q,\alpha)$ lifts to $\Omega^{\text{diff}}_d(Y,\xi)$.
This establishes injectivity.
\medskip
 
 \noindent{\it Second proof that  \eqref{eqn:sigma-prime-Y} is injective:} 
Let $Q\to Y$ be a Poincar\'e  $(d+1)$-thickening of $Y$
having stable normal bundle $\xi$. Then 
the function $t\: \cal N(Q,\xi) \to  \cal F_d(Y,\xi)$ of \eqref{eqn:t} is onto by Corollary \ref{cor:t-onto}.
Moreover, $\sigma = \sigma'\circ t\: \cal N(Q,\xi) \to L_d(\pi,w)$ is the surgery obstruction by Lemma \ref{lem:normal=quinn}.
 Hence, if $x\in \cal F_d(Y,\xi)$, there exists
$\hat x\in \cal N(Q,\xi)$ such that $t(\hat x) = x$. Then $\sigma(\hat x) = \sigma'(x) = 0$. If $\pi$ is trivial, then by 
Theorem \ref{bigthm:fund-thm} 
it follows that $\hat x$ is the trivial element. Then so is $t(\hat x) = x$. Note that
if Theorem  \ref{bigthm:fund-thm} holds for $\pi$ non-trivial, then $x$ is trivial by the same argument.
\end{proof}

\section{A Poincar\'e surgery toolbox} \label{sec:toolbox}

In  this section we
provide the tools that are used in the proof of Theorem \ref{bigthm:fund-thm}. In particular, we deal with the 
problem of performing Poincar\'e surgery below the middle dimension. 
Effort is also made
to introduce some of the homotopy theoretic results needed for Poincar\'e surgery in the non-simply connected case.

The approach taken here will be to follow the proof of Browder's fundamental theorem of simply connected surgery
\cite[thm.~II.1.1]{Browder}, which appears in
 \cite[ch.~IV]{Browder},
along the way replacing each geometric input from manifold theory with a suitable Poincar\'e space analog.

Let $f\: M^d \to X^d$ be a Poincar\'e normal map, where $X$ is connected. Then the surgery obstruction 
\eqref{eqn:Ranicki-surgery-obstruction}
\[
\sigma(f) \in L_d(\pi,w)
\]
is defined, with $\pi = \pi_1(X), w = w_1(\xi)\: \pi \to \Bbb Z_2$.

\begin{rem} When $\pi$ is the trivial group, the above $L$-groups are given by
\[
L_d = \begin{cases} 0 & d \equiv 1,3  \mod 4,\\
\Bbb Z& d \equiv 0 \phantom{13} \mod 4, \\
\Bbb Z_2 & d \equiv 2   \phantom{13}\mod 4\, .
\end{cases}
\]
There are two kinds of ingredients in Browder's approach to the fundamental theorem
of simply connected surgery in the smooth case: geometric and algebraic.
We will see that the algebraic ingredients in the Poincar\'e setting are essentially identical to those in 
Browder's text.
 
 The main geometric input
comes from Whitney's embedding theorems. 
The results of \cite{Klein_haef} and \cite{Klein_compression} 
give the analogues of Whitney's theorems in the Poincar\'e category. We state these below.
We will also  outline the proof that Poincar\'e surgery can be performed in an unobstructed way below the middle dimension.
\end{rem}

\subsection{Poincar\'e embeddings} \label{subsec:embedding} We first
define a {\it codimension zero} Poincar\'e embedding.
Suppose that $(M,\partial M)$  and $(P,\partial P)$ are  $d$-dimensional Poincar\'e pairs
and we are given a map of spaces $f\: P \to M$. Then a Poincar\'e embedding with underlying map $f$ 
is a weak equivalence of pairs
\[
(P \cup_{\partial P}  M_0,\partial M) \overset \sim \to (M,\partial M)
\]
which restricts to $f$, where  $(M_0,\partial P\amalg \partial M)$ is a Poincar\'e pair of dimension $d$, called
the {\it complement}.
In addition, we require
that  the fundamental classes of $P$ and $M_0$ are compatible with a fundamental class for $M$.
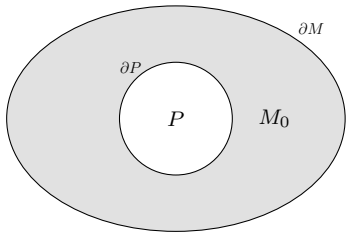
\begin{figure}
\begin{minipage}[b][5cm][s]{.75\textwidth}
\centering
\begin{tikzpicture}[scale=.75]
     \fill[white!60!gray!60] (0,0) ellipse (3 and 2);
    \draw (0,0) ellipse (3 and 2);

       \fill[white] (0,0) circle (1);
           \draw (0,0) circle (1);
    \node at (0,0) {{\tiny $P$}};
    \node at (-.8,.9) {\scalebox{.5}{$\partial P$}};
       \node at (1.75,0) {\tiny{$M_0$}};
             \node at (2.4,1.6) {\scalebox{.5}{$\partial M$}};
\end{tikzpicture}
 \caption{\tiny Schematic of a codimension zero Poincar\'e embedding of $(P,\partial P)$ in $(M,\partial M)$.}
    \label{fig:embedding}
    \end{minipage}
    \vspace{-.3in}
\end{figure}
For example, in the special case when $(P,\partial P)$ and $(M,\partial M)$ are compact smooth $d$-manifolds with boundary,
then a smooth embedding of $P$ into the interior of $M$ yields a Poincar\'e embedding in which $M_0 := M \setminus \text{int}(P)$.
For more details as well as the definition of concordance in this setting, see  \cite{Klein_compression}. 
To avoid clutter, when the Poincar\'e embedding is understood, we denote it by its underlying map
$f\: P \to M$.

Observe that codimension zero Poincar\'e embeddings can be composed: if $P \to M$ and $M \to N$ are codimension zero Poincar\'e embeddings,
then so is the composition $P \to M \to N$. If the complement of $P \to M$ is $(M_0,P \amalg \partial M)$ and the complement of $M \to N$ is 
$(N_0,\partial M \amalg \partial N)$, then the complement
of $P \to N$ is 
\[
(M_0 \cup_{\partial M} N_0,\partial P \amalg \partial N)\, .
\]

Let $K$ have the homotopy type of  a finite CW complex of dimension $k\le d-3$, and let $f\: K \to M$ be a map.
Then a {\it codimension $d-k$} Poincar\'e embedding of $f$ consists of a Poincar\'e pair $(P,\partial P)$ of dimension $d$, a homotopy equivalence
$h\:K \to P$ and a codimension zero Poincar\'e embedding $F\:P \to M$  as above such that 
$f  = F\circ h$. In addition, we require that the inclusion $\partial P \to P$ is $(d-k-1)$-connected. 
Given these data, we say that $f\: K \to M$ Poincar\'e embeds. Alternatively, as in \cite{Klein_haef}, one may define 
a codimension $d-k$ Poincar\'e embedding to be  a homotopy pushout square
\[
\xymatrix{
E \ar[r] \ar[d] & M_0\ar[d] \\
K \ar[r]_(.4)f & M ,
}
\] 
in which $E\to K$ is $(d-k-1)$-connected, $(\bar K,E)$ is a Poincar\'e pair of dimension $d$, where $\bar K$ is the mapping cylinder of $E\to K$,
and the composite homomorphism $H_{d}(M) \to H_d(M,M_0) \to H_d(\bar K,E)$ preserves fundamental classes.
The two definitions agree up to a suitable notion of concordance, since we can take $(P,\partial P) = (\bar K,E)$, and will be typically using the latter one.

\subsection{Surgery below the middle dimension}
We first explain what is meant by an elementary Poincar\'e surgery.
Let $k+\ell+1 =d$. Given a Poincar\'e space $M$ of dimension $d$, 
suppose $\phi: S^k \times D^{\ell+1} \to M$ is a  codimension  zero Poincar\'e embedding. 
That is, we have a homotopy equivalence
\[
(S^k \times D^{\ell+1} )\cup_{S^k \times S^{\ell}} M_0 \simeq M
\]
extending $\phi$
in which $(M_0,S^k \times S^\ell)$ is a Poincar\'e pair. 
We let $\phi_0 \: S^k \to M$ be the restriction
of $\phi$ to $S^k \times \{0\}$.

By {\it surgery} on $\phi$, we mean the Poincar\'e space 
\[
M' = M_0 \cup_{S^k \times S^\ell} (D^{k+1} \times S^\ell)\, . 
\]
The restriction  $\psi\: D^{k+1} \times S^\ell \to M'$ is
then a Poincar\'e embedding, and the result of surgery on 
$\psi$ recovers $M$ with the given embedding $\phi$.

The diagram
 \begin{equation} \label{eqn:cofibrations}
{\small \xymatrix{
M_0 \ar[r] \ar[d] & M \ar[r] & S^d \vee S^{\ell+1}\\
M' \ar[d]\\
S^d \vee S^{k+1}
} }
\end{equation}
is such that the vertical and horizontal rows form homotopy cofiber sequences. For the horizontal sequence, we have used the
fact that the cofiber of $M_0 \to M$ coincides with the quotient $(S^k \times D^{\ell+1})/(S^k \times S^{\ell}) \simeq S^d \vee S^{\ell+1}$, and
the vertical sequence is obtained similarly.
When $d = 2k+1$, we may apply the  singular homology functor to \eqref{eqn:cofibrations} to
 obtain the Kervaire-Milnor diagram \cite[p.~515]{Kervaire-Milnor}

\begin{equation}\label{eqn:KM-diagram}
\xymatrix{
&& H_{k+1}(M') \ar[d]_{\cdot \lambda'}\\
&& \Bbb Z \ar[d]_{\epsilon} \ar[dr]^{\lambda} \\
H_{k+1}(M) \ar[r]^(.6){\cdot \lambda} & \Bbb Z \ar[rd]_{\lambda'}\ar[r]^{\epsilon'} & H_k(M_0)\ar[r]^i \ar[d]_{i'} & H_k(M) \ar[r] & 0\\
&& H_k(M') \ar[d] \\
&& 0
}
\end{equation}
in which the horizontal and vertical sequences are exact. Hence, there is a preferred isomorphism
\[
H_k(M)/\lambda(\Bbb Z) \cong H_k(M')/\lambda'(\Bbb Z)\, .
\]
The {\it trace} of the surgery on $\phi$ is  Poincar\'e cobordism $W_\phi$ from $M$ to $M'$ given by
\[
(M\times I) \cup_{S^k\times D^{\ell+1}\times \{1\}} (D^{k+1} \times D^{\ell+1} )\, .
\] 
If we turn $W_\phi$ upside down, we get $W_{\psi}$.  Consequently, we have a chain of homotopy
equivalences
\[
M \cup_{\phi_0}  D^{k+1} \subset W_{\phi} = W_{\psi} \supset M' \cup D^{\ell+1}\, .
\]
In particular, noting that $k < \ell$ is equivalent to $k < (d-1)/2$, 
we obtain a Poincar\'e  space version of Browder's \cite[IV.1.5]{Browder}:

\begin{prop} If $k < (d-1)/2$, then $\pi_i(M) \cong \pi_i(M')$ for $i < k$ and 
\[\pi_k(M') \,\, \cong\,\,  \pi_k(M)/I_\phi\, ,\] where $I_\phi$ is the $\Bbb Z[\pi_1(M)]$-submodule of $\pi_k(M)$ generated by the image of 
the homomorphism $(\phi_0)_*\: \pi_k(S^k)\to \pi_k(M)$.
\end{prop}

Fix a Poincar\'e normal map $f \: M \to X$ of dimension $d$ and
assume we are given a codimension $\ell+1$ Poincar\'e embedding with underlying map $\phi_0\: S^k \to M$. 
Hence, 
we have a codimension zero Poincar\'e embedding $\phi\: E \to M$ given by
\[
E \cup_{\partial E} M_0 @>  \simeq >> M
\]
and a homotopy equivalence $h\:S^k @> {}_\simeq >> E$ such that  $\phi\circ h = \phi_0$. 

Assume $f\circ \phi_0\: S^k \to X$ is provided with a choice of extension to the disk $D^{k+1}$. Then we obtain
a map 
\[
F_0\: M \cup_{\phi_0} D^{k+1} \to X
\] 
extending $f$.
We seek criteria for deciding when this map can be thickened up to a normal cobordism
$F\: (W_\phi,M,M') \to (M\times I,M\times \{0\},M\times \{1\})$.
Clearly, the condition should be that there is a weak equivalence
\begin{equation}\label{eqn:e-bdye}
 (E,\partial E) \simeq (S^k \times D^{\ell+1},S^k \times S^\ell) \, .
\end{equation}
We can reformulate the condition as follows: 
if $k \le d-3$, then the homotopy fiber of $\partial E\to E$ has the homotopy type of 
$S^\ell$ (cf.\ \cite[rem.\ 2.4.1]{Klein_haef}), so $\partial E \to E$  has the fiber homotopy type of  an $\ell$-spherical fibration.
Hence, it is classified by a map $\xi\: S^k \to BG_{\ell+1}$, where $BG_{\ell+1}$ is the classifying space
of $\ell$-spherical fibrations.  

\begin{lem} \label{lem:preferred-triv} The composition $S^k @> \xi >> BG_{\ell+1} \to BG$ has a preferred null homotopy.
That is, $\xi$ has a preferred stable trivialization.
\end{lem}  

\begin{proof} There is a stable fiber homotopy equivalence $\nu_{S^k} \simeq \xi + (\phi_0)^*\nu_M$,
where $\nu_M$ is the Spivak fibration of $M$. Furthermore, $\nu_{S^k}$ has a preferred trivialization
and $\nu_M \simeq f^*\nu_X$. Consequently, we obtain
\[
\epsilon \,\, \simeq \,\,   \xi + (f\circ \phi_0)^*\nu_X 
\]
(where $\epsilon$ is trivial)
and the null homotopy of $f\circ \phi_0$ trivializes $(f\circ \phi_0)^*\nu_X$.  
This gives a stable trivialization of $\xi$. 
\end{proof}

From the lemma, we obtain a factorization 
\[
\xymatrix{
S^k \ar[r]^\xi \ar[d] &  BG_{\ell+1} \ar[d] \\
D^{k+1} \ar[r] & BG \, .
}
\]
The homotopy class of this factorization corresponds to an element $\cal O \in \pi_k(G/G_{\ell+1})$.

\begin{prop}[{cf.\ \cite[IV.1.6]{Browder}}] \label{prop:obstruction}  The obstruction $\cal O \in \pi_k(G/G_{\ell+1})$ is trivial if and only if 
$\phi_0$ extends to a Poincar\'e embedding $\phi\: S^k \times D^{\ell+1} \to M$
such that $F\: W_\phi\to X$ extends $F_0$ and 
  defines a normal bordism of  $f$.
\end{prop}

\begin{proof}  If $F$ exists then the identification \eqref{eqn:e-bdye} is valid
and so the stable trivialization of $\xi$ given above lifts in a preferred way to an unstable one. This
shows that $\cal O$ is trivial.

Conversely, suppose the obstruction $\cal O$ is represented by $S^k \to G/G_{\ell+1}$.
Let $g\: D^{k+1} \to X$ be the null homotopy of $f\circ \phi_0$. 
This amounts to the following data: a choice of $\ell$-spherical fibration over $S^k$, namely $\xi$,
together with a choice of extension of $\xi$ to $D^{k+1}$ when $\xi$ is considered as a stable spherical
fibration. The choice of this extension uses the identification 
$b\: (\phi_0)^*\nu_M\simeq (f\circ \phi_0)^\ast\nu_X$ together with $g^*\nu_M$.

If $\cal O$ is trivial, we obtain an extension $D^{k+1} \to G/G_{\ell+1}$ which amounts to
an unstable extension of $\xi$ to an $\ell$-spherical fibration over $D^{k+1}$ which lifts
the given stable extension.
In particular $\xi$ is trivializable and $(E,\partial E)$ has the homotopy type of
$(S^k \times D^{\ell+1},S^k \times S^\ell)$. Hence an extension $F\: W_\phi \to X$ is defined.
It remains  to show that $F$ is compatible with the Spivak fibration data.
Let $\eta$ be the Spivak fibration of $W_\phi$ and let $j\: M \cup_{\phi_0} D^{k+1} \subset 
W_\phi$ be the inclusion. Then $j^*\eta$ coincides with  $\nu_M$ over $M$ 
and a trivializable fibration over $D^{k+1}$ which we
can take to be $g^*\nu_X$, where along $S^k$ 
one sees that  $j^*\eta$ is identified with $(F\circ j)^*\nu_X$ and this identification coincides
with the given one on $X$.
\end{proof}

\begin{lem} \label{lem:unstable-stable} Assume $\ell \ge 2$. Then $G/G_{\ell+1}$ is $\ell$-connected. Furthermore, 
$\pi_{\ell+1}(G/G_{\ell+1}) = \Bbb Z_2$ if $\ell$ is even and $\Bbb Z$  if $\ell$ is odd.
\end{lem} 

\begin{proof}   By \cite[cor.~2.2, lem.~3.1]{James_iterated}, the map $O/O_{\ell+1}\to G/G_{\ell+1}$ is $(2\ell-1)$-connected.
 Hence,  $\pi_\ell(G/G_{\ell+1}) \cong \pi_\ell(O/O_{\ell+1})$ if $\ell \ge 2$.
Moreover, $O/O_{\ell+1}$ is $\ell$-connected.
The result now follows from \cite[thm.~IV.1.12]{Browder}.
\end{proof}

\begin{cor} \label{cor:obstruction_vanishes} 
Let $\phi_0\: S^k\to P$ be as above. If $2k+1 \le d$, then
the obstruction $\cal O$ vanishes.
\end{cor}

\begin{proof} In this case $k \le \ell$, so $\pi_k(G/G_{\ell +1})$ is trivial.\end{proof}

\begin{lem}\label{lem:unstable-stable-refine} The connecting homomorphism
\[
\partial \: \pi_{k+1}(G/G_{k+1}) \to \pi_k(G_{k+1})
\]
is injective if  $k \ne 0,1,2,6$. It is trivial if $k = 0,2,6$.
\end{lem}

\begin{rem} When $k =1$, the homomorphism $\partial \: \pi_{2}(G/G_{2}) \to \pi_1(G_2) \cong \Bbb Z$
is not injective: its kernel is cyclic of order two.
\end{rem}

\begin{proof}[Proof of Lemma \ref{lem:unstable-stable-refine}] Let $n = k+1$ and suppose $n = 1,3,7$. 
Consider the fiber sequence
 \[
 G_n @> i >> G_{n+1} @> p >> G_{n+1}/G_n\, .
 \] 
In these dimensions,  $S^{n}$ admits an $H$-space structure
$S^n \times S^n \to S^n$ which is adjoint to a map $S^n\to G_{n+1}$.
The composition $S^{n} \to G_{n+1} \to G_{n+1}/G_n$ coincides with the $(2n-3)$-connected map
 $O_{n+1}/O_n \to G_{n+1}/G_n$.  The homomorphism $p_\ast\: \pi_\ast(G_{n+1}) \to \pi_{\ast}(G/G_{n+1})$ is therefore surjective for $\ast \le 2n-3$. In particular,
 $i_\ast\: \pi_{n-1}(G_n) \to \pi_{n-1}(G_{n+1}) = \pi_{n-1}(G)$ is injective. We infer that 
 $\partial\: \pi_n(G/G_n) \to \pi_{n-1}(G_n)$ is trivial when $n =1,3,7$.

 Suppose that $n \ne 1,2,3,7$. 
Then the injectivity of $\partial$ is equivalent to the surjectivity of the homomorphism
$\pi_{n}(G_{n}) \to \pi_{n}(G_{n+2})$. 
We are indebted to Tyler Lawson for providing the following argument.

Let $F_n \subset G_{n+1}$ denote the topological monoid of basepoint preserving
self-homotopy equivalences of $S^{n}$. Then the double suspension defines a homomorphism 
$F_{n} \to F_{n+2}$. 

Unreduced  suspension defines a homomorphism $G_n \to F_n$
which is $(2n-4)$-connected \cite[cor.~2.2]{James_iterated}.\footnote{The reader should note
James uses a different indexing convention: his $G_{n}$ is our $G_{n+1}$. 
Our convention is the more frequent one appearing in the literature.} 
Consequently, if $n \ge 4$,
the homomorphism $\pi_{n}(G_{n}) \to \pi_{n}(G_{n+2})$ is surjective if and only if the homomorphism
$\pi_{n}(F_{n}) \to \pi_{n}(F_{n+2})$ is surjective. 

Since $\pi_j(F_n) \cong \pi_{j+n}(S^n)$ for $j \ge 1$, it will suffice to show
 that double suspension homomorphism
\[
E^2\: \pi_{2n}(S^n) \to \pi_{2n+2}(S^{n+2}) 
\]
is surjective.
We will show that the latter is surjective when $n \ne 1,3,7$.
Consider the EHP-sequences for $S^n$ and $S^{n+1}$:
{\small \begin{align*}
  &  \cdots @> P >>    \pi_{2n}(S^n) @> E >>       \pi_{2n+1}(  S^{n+1}) @> H >>    \Bbb Z @> P >> \pi_{2n-1}(S^{n}) @> E>>  \pi_{2n}(S^{n+1}) \to 0 \, ,\\
     \cdots @> P >> &   \pi_{2n+2}(S^{n+1}) @> E >>       \pi_{2n+3}(  S^{n+2}) @> H >>    \Bbb Z @> P >> \pi_{2n+1}(S^{n+1}) @> E>>  \pi_{2n+2}(S^{n+2}) \to 0 
\end{align*}}
\!\!\cite[thm.~1.1]{James_iterated}.
If $n$ is even then $H\: \pi_{2n+1}(S^{n+1}) \to  \Bbb Z$ is trivial because $\pi_{2n+1}(S^{n+1})$ is finite \cite{Serre}.

If $n$ is odd, then $P\:  \Bbb Z \to \pi_{2n+1}(S^{n+1})$ sends the generator to the Whitehead product $[\iota,\iota]$.
Furthermore, $H\circ P\:\Bbb Z \to  \Bbb Z$ is multiplication by $2$, since $H([\iota,\iota])$ is twice the generator.

Consequently, given $\alpha \in \pi_{2n+2}(S^{n+2})$,  we can lift it to $\beta\in \pi_{2n+1}(S^{n+1})$ and any
other lift is of the form $\beta + j[\iota,\iota]$ for $j\in \Bbb Z$. This lifts further to $\pi_{2n}(S^n)$ if and only if
$H(\beta + j[\iota,\iota]) = H(\beta) + 2j = 0$, i.e., $H(\beta)\in \Bbb Z$ is even. This is only possible when
$\pi_{2n+1}(S^{n+1})$ has no element of Hopf invariant one, which according to \cite{Adams} only occurs 
when $n \ne 1,3,7$. 
\end{proof}

The following result is the Poincar\'e duality space analog of \cite[thm.\ IV.1.13]{Browder} and
\cite[thm.~1.2]{Wall}.

\begin{thm} \label{below_middle} Given $X\in \Top_{/BG}$,  let $f\: M \to X$ be a 
Poincar\'e space over $X$ of dimension $d \ge 5$.
Then $f$ admits a cobordism over $BG$ to $f'\: M' \to X$ which  is $\lfloor d/2\rfloor$-connected.
\end{thm}

\begin{proof} 
By implementing the obstruction $\cal O$ above, Browder's argument  \cite[thm.\ IV.1.13, lem.~IV.1.13-16]{Browder} 
adapts with just one additional change: 
wherever the Whitney embedding theorem is used one replaces it with part (i) of
Corollary \ref{cor:easy} below.  \end{proof}

\subsection{Poincar\'e embedding results required for surgery} 
The Poincar\'e space version of the embedding theorem alluded to in Theorem \ref{below_middle} 
is a consequence of the following result.

\begin{thm}[{\cite[cor.~G]{Klein_compression}}, {\cite[th.\ A]{Klein_haef}}]
\label{my_thm}  Let 
$K$ have the homotopy type of a connected CW complex of dimension $\le k$ and let $M$ be a connected Poincar\'e space of dimension $d$,
with $k \le d-3$. 
Let $f\: K \to M$ be a map. Assume one of the following conditions holds:
\begin{enumerate}[(i).]
\item $f$ is $(2k-d+2)$-connected, or
\item $f$ is $(2k - d+1)$-connected and $3k+4 \le  2d$.
\end{enumerate}
Then $f\: K\to M$ underlies a Poincar\'e embedding. 
 \end{thm}
 
\begin{rem} When $M$ is smooth, then in (i) it is enough that $f$ is
$(2k-d+1)$-connected \cite[p.~76]{Wall_thickening}.
\end{rem}

 \begin{cor}\label{cor:easy} Let 
$K$ be a connected Poincar\'e space of dimension $k$ and let $M$ be a connected Poincar\'e space of dimension $d$.
Let $f\: K \to M$ be a map. Then $f$ underlies a Poincar\'e embedding if
\begin{enumerate}[(i).]
\item  $d\ge \max(5,2k+1)$, or
\item  $d \ge \max(7,2k)$ and $M$ is $1$-connected.
 \end{enumerate}
 \end{cor} 
 
 \begin{proof} If $d\ge \max(5,2k+1)$ then $f$ is $(2k - d+1)$-connected since $f$ is $0$-connected.
The inequality $d\ge \max(5,2k+1)$ also implies
 $3k+4 \le 2d$ and $k\le d-3$. Then a Poincar\'e embedding exists by part (ii) of Theorem \ref{my_thm}.
 
 If $d \ge \max(7,2k)$ and $M$ is $1$-connected, then $f$ is $1$-connected and $2k-d+1 \le 1$.
So $f$ is $(2k - d+1)$-connected. Moreover, $d \ge \max(7,2k)$ guarantees that
$3k+4 \le 2d$ and $k\le d-3$. Then a Poincar\'e embedding exists by part (ii) of Theorem \ref{my_thm}.
 \end{proof}

\begin{rem}    In the smooth case (cf.~Browder \cite[p.\ 95]{Browder}), the inequality in (ii) of Corollary \ref{cor:easy}
 can be improved to $d \ge \max(5,2k)$. 
 \end{rem}
 
 We next describe an embedding theorem applicable to middle dimensional surgery.
Let $M$ be a $1$-connected Poincar\'e space of dimension $d= 2k$ with Spivak fibration $\xi$.
 Consider pairs $(g,\tau)$ in which $g\: S^k \to M$ is a map and $\tau$ is a stable trivialization of
 $g^\ast\xi$. Two such pairs $(g_0,\tau_0)$ and  $(g_1,\tau_1)$ are deemed equivalent if there is a homotopy
 $G\: S^k \times I\to M$ from $g_0$ to $g_1$ and a trivialization  of $G^\ast \xi$ extending $\tau_i$.
If $k \ge 2$, then the set of such equivalence classes is equipped with the structure of an abelian group \cite[pp.~1--2]{Klein_sphere}.
  We  denote it by $\frak F_k(M)$.
  
Set 
\[
Q_k = \Bbb Z/(1-(-1)^{k})\Bbb Z\, .
\]
Then $Q_k = \Bbb Z$  if $k$ is even and $Q_k = \Bbb Z_2$ if $k$ is odd.
 
 \begin{thm}[{\cite[thm.~A]{Klein_sphere}}] \label{thm:Klein_sphere} Assume in addition $k > 3$.\footnote{This corrects the statement of \cite[thm.~A]{Klein_sphere}, where it is asserted that
 $k >2$ suffices. The  author failed to take account of the condition $3k+4\le 2d$.}
   Then there is a function
 \[
\mu\:  \frak F_k(M) \to Q_k
 \]
 such that $\mu(x) = 0$ if and only if $x$ is represented by a codimension zero Poincar\'e embedding
 $S^k \times D^k \to M$. Moreover, $\mu$ satisfies
 \begin{enumerate}[(i)]
 \item $\mu(x+y) = \mu(x) + \mu(y) + \lambda(x,y)$,
 \item $\lambda(x,x)= (1+(-1)^k)\mu(x)$,
 \item $\mu(ax) = a^2\mu(x)$, for $a\in \Bbb Z$,
 \end{enumerate}
where $\lambda(x,y)$  is the value of the intersection pairing applied to the Hurewicz images of $x$ and $y$ in $H_k(M)$.
 The intersection pairing is defined as Poincar\'e dual to the cup product pairing.
  \end{thm}

\begin{rem}[Definition of $\mu$] We outline here the definition of the invariant $\mu$ and refer the reader to \cite{Klein_sphere} 
for the details. Set $(P,\partial P) = (S^k \times D^k, S^k \times S^{k-1})$.
It is explained in \cite[\S3]{Klein_sphere} that any $x\in \frak F_k(M)$
is represented by a codimension zero Poincar\'e embedding $P\times  D^j \to M \times D^j$ for $j$ sufficiently large.
One then has an associated Pontryagin-Thom collapse \cite[\S4]{Klein_sphere}
\[
\hat t(x) \in \{M_+, P/\partial P\}\, ,
\]
where if $A$ and $B$ are based spaces, recall that
 $\{A,B\}$ denotes the abelian group of stable homotopy classes of maps $A\to B$.
Let $H \: \{A,B\} \to \{A,D_2(B)\}$ denote the stable Hopf invariant \cite{klein2026stablehopf}, \cite{Klein-Naef}, \cite{Crabb-Ranicki}, \cite{Milgram}, where
\begin{equation} \label{eqn:quadratic-construction}
D_2(B) := (B\smsh B) \smsh_{\Bbb Z_2} (E\Bbb Z_2)_+
\end{equation}
is the quadratic construction of  $B$.
Then
\[
\mu(x) = H(\hat t(x)) \in \{M_+, D_2(P/\partial P) \} \cong Q_k
\]
\cite[\S5]{Klein_sphere}. 
\end{rem}

\begin{rem} Let $E\: [A,B] \to \{A,B\}$ denote stabilization, i.e., the operation which considers an unstable map as a stable map. Then $H\circ E$ is trivial. Conversely,
suppose $f\in \{A,B\}$ and assume that $B$ is $r$-connected and $A$ is
a CW complex of dimension $\le 3r+1$. Then $H(f) = 0$ implies $f = E(f')$ for some $f'\in [A,B]$ \cite[II\S8]{Milgram}.
The condition $\dim A \le 3r+1$ is called the {\it metastable range}.
 \end{rem}

Suppose that $f\: M^{2k} \to X^{2k}$ is a $k$-connected normal map of Poincar\'e spaces,  $k \ge 4$. Assume that $X$ is $1$-connected. Let $K_k(M) := \pi_{k+1}(f)$ denote the surgery kernel. Then there is a forgetful homomorphism $K_k(M) \to \frak F_k(M)$. 

\begin{defn}\label{defn:mu}
By slight abuse of notation, the quadratic function defined by the composition
\[
K_k(M) \to \frak F_k(M) @> \mu >> Q_k
\]
is also denoted by $\mu$.
\end{defn}

\subsection{Relation to other invariants} 
\subsubsection{Wall's double point invariant} \label{subsubsec:double-point} Suppose $f\: M^{2k} \to X^{2k}$ is a $k$-connected {\it classical} normal map, i.e.,  the source
$M$ is a closed smooth manifold, $X$ is a 1-connected Poincar\'e complex, and $f$ is a map over $BO$. In this setting, Wall \cite[ch.~5]{Wall} defines a quadratic function 
\[
\mu^W \: K_k(M) \to Q_k\
\] 
geometrically as follows:\footnote{To distinguish between Wall's double point invariant and the 
homotopical one defined here, I have modified Wall's notation by indicating his invariant with a superscript.}
Represent an element of $x\in K_k(M)$ by a self-transverse smooth immersion $\phi\: S^k \to M$. Then $\phi$ has no triple points
and the set of self-intersections of $\phi$ is a finite set labeled by elements of $Q_k$. 
Consequently, the sum of these self-intersections defines an element of $Q_k$ which
is independent of the choice of representative $\phi$.

By \cite[\S1]{Koschorke-Sanderson}, \cite[pp.~236-237]{Ranicki}, or \ref{subsubsec:mu-R} below  (see also \cite[thm.~6.19,prop.~6.20]{Crabb-Ranicki}, \cite[pp.~279-282]{Ranicki}),
the invariant $\mu^W$ coincides with $\mu$ as defined above. 
Consequently, the geometrically defined $\mu^W$ admits an alternative homotopical description, namely  $\mu$. (Note: The identification of $\mu$ with $\mu^W$
 relies on manifold transversality. However, transversality is not needed to define $\mu$.)

\subsubsection{Ranicki's quadratic self-intersection invariant} \label{subsubsec:mu-R} Let $f\: M^{2k} \to X^{2k}$ be a classical normal map as in
\ref{subsubsec:double-point}.  Given a smooth immersion $g\: S^k \to M$ with trivial normal bundle, Ranicki describes a
regular homotopy invariant 
\[
\mu^R(g) \in Q_k
\] 
\cite[p.~235]{Ranicki}. The latter is defined by evaluating the composition
\begin{equation}\label{eqn:mu-R}
H_{2k}(M) @>  \psi_G >> H_{2k}(D_2(P/\partial P)) @> v^r >> \hom(H^k(P/\partial P), Q_k) \cong Q_k
\end{equation}
at the fundamental class $[M]$. Here, $\psi_G$ denotes Ranicki's algebraic quadratic construction \cite[prop~1.5]{Ranicki}, $v^r$ is the quadratic Wu class
\cite[p.~114]{Ranicki}, and $(P,\partial P) = (S^k \times D^k,S^k \times S^{k-1})$. In particular, $P/\partial P$ is the Thom space of the normal bundle of $g$.

In the above, we have taken the liberty of identifying
$H_{2k}(D_2(P/\partial P))$ with Ranicki's $Q$-group $Q_{2k}(C_\ast(P/\partial P))$ \cite[p.~93]{Ranicki-1}.
The isomorphism in the display follows from 
the fact that $H^k(P/\partial P) \cong \Bbb Z$.  Furthermore, it is straightforward to check that the homomorphism $v^r$ is an isomorphism.
By \cite[p.~204, bottom]{Ranicki}, $\psi_G$ coincides with the homomorphism
on homology defined by the composition $H\circ \hat t$, where $\hat t$ is the Pontryagin-Thom construction and $H$ is the stable Hopf invariant.
It follows that $\mu^R(g) = \pm \mu(g)$. Without loss in generality, we may assume that $\mu^R = \mu$.

More generally, if $f\: M^{2k} \to X^{2k}$ is a Poincar\'e normal map, then it is immediate that  $\mu^R \: \frak F_k(M) \to Q_k$ may be defined using \eqref{eqn:mu-R}, and 
it is clear from the above discussion that $\mu^R = \mu$.

\subsubsection{Browder's invariant $\cal O$} \label{subsubsec:Browder-O} 
Let $f\: M^{2k} \to X^{2k}$ be a $k$-connected classical normal map, with $k \ge 3$, and $X$ $1$-connected. 
Then any element of $x\in K_k(M)$ is represented by a smooth embedding $\phi\: S^k \to M$.
The obstruction $\cal O(x)$ to finding a representative $\phi$ which is framed 
lies in the group $\pi_k(O/O_k)\cong Q_k$ (see \cite[thm.~IV.1.6]{Browder}).

By analogy, suppose now that  $f\: M^{2k} \to X^{2k}$ is a Poincar\'e normal map, $k\ge 4$, with $X$ $1$-connected.
By Corollary \ref{cor:easy},  any element of $x\in K_k(M)$ is represented by a
 codimension $k$ Poincar\'e embedding,
say
\[
\xymatrix{
E \ar[r] \ar[d]_\xi  & C\ar[d] \\
S^k \ar[r]_{\phi}  & M,
}
\]
where $\xi$ is a $(k-1)$-spherical fibration equipped with a preferred stable trivialization.
Hence, the obstruction to trivializing $\xi$ is given by an element
 $\cal O(x) \in \pi_k(G/G_k) \cong Q_k$ (cf.~Lemma \ref{lem:unstable-stable}). If $\cal O(x)$ is trivial, then $x$
is represented by a framed Poincar\'e embedded $k$-sphere in $M$.

The assignment $x\mapsto \cal O(x)$ defines
another function
\[
\cal O \: K_k(M) \to Q_k\, .
\]
Consequently, both $\mu$ and $\cal O$ provide the complete obstruction to representing $x$ by a framed
Poincar\'e embedded sphere. I claim
that $\mu = \cal O$.

When $k$ is odd, the claim follows  from the fact that $Q_k = \Bbb Z_2$.
If $k$ is even, then $2\mu(x) = \lambda(x,x) \in \Bbb Z$ by Theorem \ref{thm:Klein_sphere}.
Arguing just as in \cite[lem~IV.4.5]{Browder}, it follows that
$2\cal O(x)$ also coincides with $\lambda(x,x)$. We conclude that  $\mu(x) = \cal O(x)$ when $k$ is even.

\subsection{Three approaches to the surgery obstruction}
The surgery obstruction of a Poincar\'e normal map 
   \eqref{eqn:Ranicki-surgery-obstruction} was defined in \S\ref{sec:bordism}  using Ranicki's machinery.
   Here we define alternative versions which will be seen to coincide.

\subsubsection{A Wall-type approach} \label{subsubsec:Wall-approach} 
Wall  defines the $L$-group $L_{2k}$ of the trivial group as 
the Witt group of  $(-1)^k$-hermitian forms 
\[
(\Lambda,\lambda,\mu)
\]
over $\Bbb Z$ (see the proof of Theorem \ref{thm:Wall_realization} below for the definition), where stability is with respect to hyperbolic forms.
There is a canonical isomorphism $L_{2k}\cong Q_k$ given by
\begin{equation} \label{eqn:L-group}
(\Lambda,\mu,\lambda) \mapsto \begin{cases}
\tfrac{I(\lambda)}{8}, & k \text{ even},\\
A(\mu) &  k \text{ odd},
\end{cases}
\end{equation}
where $I(\lambda)$ is the signature of $\lambda$ and $A(\mu)$ is the Arf invariant of $\mu$.

If  $f\: M^{2k} \to X^{2k}$ is a $k$-connected classical  normal map, then Wall defines the surgery obstruction as the equivalence class
\[
\sigma^W(f) := [K_k(M),\lambda,\mu^W] \in L_{2k},
\]
where $\lambda$ is the intersection form on $K_k(M)$ and $\mu^W$ is the geometric
double point obstruction. Taking  \ref{subsubsec:double-point} into account, we are entitled to 
redefine Wall's surgery obstruction as the equivalence class of $ (K_k(M),\lambda,\mu)$, where $\mu$ is the homotopically defined obstruction
appearing  in Definition \ref{defn:mu}.

By analogy, we have the following.

\begin{defn}\label{defn:Wall-surgery-obstruction}
Given  a $k$-connected Poincar\'e  normal map $f\: M^{2k} \to X^{2k}$, with $X$  $1$-connected,
the {\it Wall surgery obstruction} of $f$ is
the equivalence class in $L_{2k}$ of the $(-1)^k$-hermitian form over $\Bbb Z$ defined by
\[
\sigma^W(f) := (K_k(M),\lambda,\mu)\, .
\]
\end{defn}

\subsubsection{A Browder-type approach} Here we follow the approach of \cite[ch.~III]{Browder}.
 Let $f\: M\to X$ be a normal map of $d$-dimensional Poincar\'e spaces, $d \ge 7$ or $d = 5$, 
 where $X$ is $1$-connected.  As in the smooth case, Browder's surgery obstruction  $\sigma^B(f) \in L_d$  is defined as follows: 
\medskip

\noindent {\it The case $d = 2k+1$:} In this case $\sigma^B(f) = 0$.
\medskip

\noindent {\it The case $d = 2k$, with $k$ even:} We set 
   \[
   \sigma^B(f) = \tfrac{I(f)}{8} := \tfrac{I(M) - I(X)}{8}\, ,
   \]
   where $I(M)$ is the signature of the intersection form
   \[
\lambda_M\: H_{k}(M;\Bbb Q) \otimes H_{k}(M;\Bbb Q) \to \Bbb Q\, .
   \]   
   To make sense of this, note that $M$ comes equipped with a preferred orientation by our definition of normal map. 
   Next, as $d \equiv 0 \mod 4$, and Poincar\'e duality holds,
 $\lambda_M$  is a  non-singular symmetric bilinear form.
After choosing a basis for $H_{k}(M;\Bbb Q) $, we may identify $\lambda_M$ with
a symmetric matrix $A$ with determinant $\pm 1$.
  Then $I(f)$ is defined as the index of $A$, i.e., the dimension of the positive eigenspace  minus the dimension of the negative
  eigenspace. 
In addition, $I(f)$ is divisible by $8$ by \cite[thm.~III.3.9]{Browder} (note: all the results of \cite[III.3]{Browder} are stated for Poincar\'e duality spaces).
When $f$ is $k$-connected, the invariant $I(f)$ coincides with the signature of the bilinear form
\[
\lambda' \: K_k(M;\Bbb Q)\otimes K_k(M;\Bbb Q) \to \Bbb Q\, .
\]
given by restricting $\lambda$ to the rational surgery kernel.
\medskip

\noindent {\it The case $d = 2k$, with $k$ odd:} There are various equivalent 
ways to define the obstruction $\sigma^B$ in this case. One may use Browder's quadratic function $\psi$  which uses  
 functional Steenrod squares \cite[III.4]{Browder}.  In the case of normal maps in the classical sense,
  $\psi$ is identified with the function $\cal O$ of \cite[th.~IV.4.1]{Browder}, the latter which is 
  the smooth analog of what we have described above. 
Consequently,  by \ref{subsubsec:Browder-O}, we can use the quadratic function
$\mu$ of Definition \ref{defn:mu}, since $\cal O = \mu$.

Assume that $f$ is $k$-connected and  set $V := K_k(M) \otimes \Bbb Z_2$. 
 Then $V$
 is a finite dimensional $\Bbb Z_2$-vector space and $\mu$  descends to a non-singular quadratic form 
\begin{equation}\label{eqn:mu}
 \mu\: V\to \Bbb Z_2\, .
\end{equation}
  We define $\sigma^B(f) \in \Bbb Z_2$ as the {\it Arf invariant} of $\mu$. It is equal to
  1 if and only if $\mu$ sends a majority of elements of $V$ to $1\in \Bbb Z_2$ (cf.~\cite[prop.~3.1.8]{Browder}).

\subsubsection{Comparison with the surgery obstruction of \S\ref{sec:bordism}} Let $f\: M^{2k} \to X^{2k}$ be a $k$-connected classical normal map, $X$ 1-connected.
Let
\[
\sigma^R(f) := \sigma(f)
\] be
the surgery obstruction  defined in \eqref{eqn:Ranicki-surgery-obstruction}. 
Then by \cite[prop.~5.4]{Ranicki} we have $\sigma^W(f)  = \sigma^R(f) $.

Inspection of the proof of  \cite[prop.~5.4]{Ranicki}  shows that it applies more generally to
highly connected Poincar\'e normal maps. Indeed, the only  manifold-theoretic input appearing in the proof is the identity $\mu^W(g) = \mu^R(g)$ in the case of smooth framed immersions $g\: S^k \to M$.  Recall above that $\sigma^W$ in the Poincar\'e  case uses  $\mu$ in place of Wall's $\mu^W$, and therefore the 
analogous identity is $\mu = \mu^R$ in the Poincar\'e case; the latter identity was explained in  \ref{subsubsec:mu-R}.

Consequently, we have
\[
\sigma^R(f) = \sigma^W(f)
\]
for $k$-connected Poincar\'e normal maps $f\: M^{2k} \to X^{2k}$. 
In addition, the isomorphism \eqref{eqn:L-group}  shows that $\sigma^W(f) = \sigma^B(f)$ when $X$ is $1$-connected.

Summarizing, we conclude:

\begin{cor} \label{cor:all-agree}
In the simply connected case, for $k$-connected Poincar\'e normal maps $f\: M^{2k} \to X^{2k}$, with $k \ge 4$, we have
\[
\sigma(f) = \sigma^W(f) = \sigma^B(f) \in L_{2k} = Q_k\, .
\]
\end{cor}

%
%

 \section{Proof of Theorem \ref{bigthm:fund-thm} for $d$ odd}\label{sec:d-odd}

   Let $f\: M^d \to X^d$ be a Poincar\'e  normal map,  $d = 2k+1$, where $X$ is $1$-connected and $k \ge 2$. 
    By Theorem  \ref{below_middle}, we may assume that $f$ is $k$-connected since by
    Lemma \ref{lem:Ranicki-surgery-obstruction} $\sigma(f)$ is a normal cobordism invariant.
  Let $K_k(M)$ be the surgery kernel.
   
 Following the proof of  \cite[thm.~IV.3.1]{Browder},
write $K_k(M) \cong T \oplus F$, in which $F$ is finitely generated free abelian and
  $T$ is the finite torsion subgroup. Let $x\in F$ generate an infinite cyclic summand.
  By Proposition \ref{prop:obstruction} and Corollary \ref{cor:obstruction_vanishes},  $x$ is realized by a codimension zero Poincar\'e embedding 
  $S^k \times D^{k+1} \to M$.
   Then  the result of surgery on the latter results in a Poincar\'e normal map $f'\: M' \to X$ such that the rank of  $K_k(M')$ is strictly less than
   the rank of $K_k(M)$ and the torsion subgroup is unchanged. Iterating this procedure finitely many times, 
   we may assume without loss in generality that $K_k(M) = T$ is a finite abelian group. The argument now proceeds
   in cases depending
   on the parity of $k$.

    \subsection{The case $d \equiv 1 \mod 4$}\label{sec:d-odd-p-even} In this instance $k$ is even and
the  argument proceeds exactly as in \cite[p.~104]{Browder}. 
Let $x\in K_k(M)$ generate a cyclic summand. The effect of surgery in this case results
   in a Poincar\'e normal map $f_1\: M_1 \to X$, in which $K_k(M_1)$ has torsion group $T_1$  isomorphic to a proper subgroup of $T$.
   Then we apply the previous procedure to $f_1$ to obtain a $k$-connected normal cobordism to 
     $f_2\: M_2\to X$ with $K_k(M_2) = T_2$, where $T_2$ is isomorphic to a proper subgroup of $T_1$. Iterating these steps, eventually one obtains a Poincar\'e  normal cobordism 
     to a map $f_j\: M_j \to X$ where $K_k(M_j) = 0$, i.e.,  $f_j$ is a homotopy equivalence.
     It follows that $f\: M \to X$ is normally cobordant to a homotopy equivalence.

   \subsection{The case $d \equiv 3\mod 4$}\label{sec:d-odd-p-dd}
   In this instance $k$ is odd.
Given an element of $x\in K_k(M)$ we represent it by a
 codimension zero Poincar\'e embedding  $\phi\: S^k \times D^{\ell+1} \to M$ 
where $k = \ell$. So we have
a  homotopy pushout square
   \[
   \xymatrix{
   S^k \times S^\ell \ar[r]\ar[d]_{\cap}& M_0\ar[d] \\
   S^k \times D^{\ell+1} \ar[r] & M\, .
   }
   \]
The trace of the surgery on $\phi$
defines a normal cobordism
 \[
  (W_\phi,M,M') \to (X \times I,X\times \{0\}, X\times \{1\})
  \] 
  of the normal map $f$.

If $\omega\: S^k \to G_{\ell+1}$ is any map, then we may take its adjoint to obtain  a self-homotopy equivalence of pairs
  $\hat\omega\: (S^k \times D^{\ell+1}, S^k \times S^\ell) \to (S^k \times D^{\ell+1}, S^k \times S^\ell)$.
  Precomposing the square with $\hat \omega$, one obtains a new Poincar\'e embedding
  $\phi_{\omega}$. By  Proposition \ref{prop:obstruction},
   surgery  on $\phi_\omega$ defines a normal cobordism of $f$ if and only if
  the homotopy class of the composition 
  \[
  S^k@> \omega >>G_{\ell+1} @>>> G
  \] 
  is trivial, i.e.,  on homotopy classes  $[\omega] = \partial(j)$ for some
  $j \in \pi_{k+1}(G/G_{\ell+1})$. 
    By Lemmas \ref{lem:unstable-stable} and \ref{lem:unstable-stable-refine},  
   $\partial \: \Bbb Z \cong \pi_{k+1}(G/G_{\ell+1}) \to \pi_k(G_{\ell+1}) $
   is a monomorphism (here we remind the reader that $k = \ell$ is odd). Hence, we may identify $\omega$ with an integer $j$. 
 Summarizing, each $x\in K_k(M)$ can be represented
   by countably many codimension zero Poincar\'e embeddings $\phi_j$ for $j \in \Bbb Z$, each 
 defining a normal cobordism of $f$.\footnote{More precisely, the set of concordance classes of
 codimension zero Poincar\'e embeddings representing $x$ is a $\Bbb Z$-torsor.}

 The rest of the proof follows the line of argument of  \cite[pp.~106-107]{Browder}. Namely, 
let $q$ be the largest prime that divides the order of $T = K_k(M)$.
Let $x\in K_k(M)$ be an element having non-trivial reduction mod $q$.
Then we may find a Poincar\'e embedding $\phi\: S^k \times D^{\ell+1} \to M$
such that the restriction $S^k \times \{0\} \to M$ represents $x$. Surgery on $\phi$ defines
a normal cobordism of $f$ to a normal map $f'\: M' \to X$. 
Then $K_k(M;\Bbb Z_q) := K_k(M)\otimes \Bbb Z_q$ has the structure of a finite dimensional vector space
over the field $\Bbb Z_q$.

By straightforwardly adapting  \cite[prop.~IV.3.12]{Browder} to the Poincar\'e context, 
one may choose $\phi$ (i.e., using the  $\phi_j$  above) so that
 the torsion subgroup of $T' \subset K_k(M')$ satisfies 
 \[
 |T'| \le |T| \quad \text{ and } \quad 
 \dim K_k(M';\Bbb Z_q) < \dim K_k(M;\Bbb Z_q)\, .
 \]
If we iterate this procedure, we find after finitely many steps that $f$ is normally cobordant to 
a $k$-connected Poincar\'e normal map $f_1\: M_1 \to X$ such that the torsion subgroup $T_1 \subset K_k(M_1)$
satisfies $|T_1| \le |T|$ and $\dim K_k(M_1;\Bbb Z_q) = 0$. Consequently, $T_1 = K_k(M_1)$ is a torsion group 
 whose torsion is prime to  $q$
whose order is strictly smaller than the order of $K_k(M)$.

If we substitute $f$ by $f_1$ and repeat  the above recipe, then after finitely many steps we obtain a normal cobordism of $f$ to a Poincar\'e normal map $f_n\: M_n \to X$
such that $K_k(M_n)= 0$, i.e., $f$ is normally cobordant to an equivalence. \qed

 \section{Proof of Theorem \ref{bigthm:fund-thm} for $d$ even}\label{sec:d-even}
     Let $f\: M \to X$ be a normal map of Poincar\'e spaces having dimension  $d =2k \ge 8$. Assume
     $X$ is $1$-connected, and $\sigma(f) =0$. 
     
     We may assume that $f$ is $k$-connected without changing $\sigma(f)$
     by doing surgery below the middle dimension, since $\sigma(f)$ is an invariant of Poincar\'e normal bordism
     by Lemma \ref{lem:Ranicki-surgery-obstruction}. By Poincar\'e duality (cf.~\cite[I.2.6]{Browder}), it follows that
     $K_k(M)$ is a finitely generated torsion free abelian group.
     
     We consider two subcases.
     
   \subsection{The case $d \equiv 0 \mod 4$} 
   Assume  $k$ is even.  Then $0 = \sigma(f) = I(f)/8$, where $I(f)$ is the signature of the intersection form
   $\lambda$ on $K_k(M)$. Note that $\lambda(x,y) = \mu(x+y) - \mu(x) - \mu(y)$.
   Since $I(f) = 0$, $\lambda$ is indefinite. Hence, there is an indivisible
 element $x\in K_k(M)$ such that $\mu(x) = 0$ (cf.~\cite[prop.~3.1.3]{Browder}, \cite{Milnor_man-quad}).
 Then we may represent $x$ by a codimension zero Poincar\'e embedding $\phi\: S^k \times D^{k} \to M$. Surgery on $\phi$
produces a normal cobordism of $f$ to a $k$-connected normal map $f'\: M' \to X$ 
such that the rank of $K_k(M')$ is strictly less than that of  $K_k(M)$. The last assertion
is a formal consequence of the commutative 
diagram
\[
{\small \xymatrix{
            & \Bbb Z \ar[d]_{[S^k \times \{0\}]}  \ar[dr]^{n \mapsto nx} \\
0 \ar[r] & H_k(M_0) \ar[r] \ar[d] & H_k(M) \ar[r]^(.6){x\cdot}  & \Bbb Z  \ar[rr]^(.4){[0 {\times} S^{k-1}]} && H_{k-1}(M_0) \ar[r] & H_{k-1}(M) \ar[r] & 0\\
 & H_k(M') \ar[d] \\
& 0
}}
\]
in which the vertical and horizontal sequences are exact. This diagram 
arises by applying homology to the diagram \eqref{eqn:cofibrations} 
 (cf.~\cite[cor.~IV.2.1]{Browder}, \cite[p.~527]{Kervaire-Milnor}). 

Since $f$ and $f'$ are normally cobordant, we have $I(f') = 0$. Iterating, after a finite number of steps
we obtain a Poincar\'e normal map $f_1\: M_1 \to X$ with trivial $K_k(M_1)$. Hence, the normal map $f_1$ is an equivalence.
   
  \subsection{The case $d \equiv 2 \mod 4$} 
Let $d = 2k$ with $k$ odd. Assume $\sigma(f) = 0$ and recall that $\sigma(f)$ is the Arf invariant of the quadratic form 
$\mu\: K_k(M;\Bbb Z_2) \to \Bbb Z_2$. 

Recall that a symplectic basis for a  nonsingular skew symmetric bilinear form $\lambda\: V \otimes V\to \Bbb Z_2$ 
consists of a basis for $V$ with basis elements
$(e_i,f_i)$ with $i = 1,\dots r$ such that 
\begin{equation} \label{eqn:symplectic-arf}
\lambda(e_i,e_i) = 0 = \lambda(f_i,f_i), \quad \lambda(e_i,f_j) = \delta_{ij}\, .
\end{equation}
If $\lambda$ has a refinement to a quadratic form  $\mu \: V\to \Bbb Z_2$
in the sense that  $\lambda(x,y) = \mu(x+y) + \mu(x) + \mu(y)$, then
it is an algebraic fact that the Arf invariant of $\mu$ is given by 
\[
\sum_{i=1}^\ell \mu(e_i)\mu(f_i) \, .
\]
Moreover, the Arf invariant is trivial if and only if one can
find a symplectic basis for $V$ satisfying $\mu(e_i) = \mu(f_i) = 0$, for  $i = 1,\dots,r$ (cf.~\cite[p.~54]{Browder}).

In our case, $\lambda$ is the intersection form on $K_k(M;\Bbb Z_2)$, and $\mu\: K_k(M;\Bbb Z_2) \to \Bbb Z_2$
is the quadratic form \eqref{eqn:mu}. Since $\sigma(f) = 0$, there is indeed a symplectic basis
$\{e_i,f_i\}$ of $K_k(M;\Bbb Z_2)$ such that $\mu(e_i) = \mu(f_i) = 0$.
The argument now proceeds just as in \cite[lem.~8.4]{Kervaire-Milnor}:
We represent $e_r$ by a Poincar\'e embedding $\phi\: S^k \times D^k \to M$
and do surgery on it to obtain a cobordism to a normal map $f_1\: M' \to X$.
Then $K_k(M';\Bbb Z_2)$ admits a symplectic basis of the form
$(e_i,f_i)$ with $i = 1,\dots r-1$ which again satisfies condition \eqref{eqn:symplectic-arf}.
It follows that $\sigma(f') = 0$. Iterating the procedure a further $r-1$ times, we obtain
a normal cobordism to a Poincar\'e normal map $f'\: M_1\to X$ with $K_k(M_1) = 0$.\qed

\section{Wall realization} \label{sec:realization}

The goal of this section is to prove a Poincar\'e space analog of the 
Wall realization theorem in even dimensions \cite[thm.~5.8]{Wall}. 
The odd dimensional case is deferred to another paper.

Let $N$ be a connected Poincar\'e space of dimension $d-1$ with $d =2k  \ge 8$.
Let $\sigma \in L_d(\pi,w)$ be any element, where $\pi = \pi_1(N), w = w_1(\xi)$.

\begin{thm}\label{thm:Wall_realization} 
With respect to these assumptions, there is a Poincar\'e cobordism
$(W,\partial W)$  of dimension $d$
with $\partial W = \partial_0 W\amalg \partial_1 W$,
$\partial_0 W = N$, and a normal map 
\[
(f,f_0,f_1)\: (W,\partial_0 W,\partial_1 W)  \to (N \times I,N\times \{0\},N\times \{1\})
\]
such that $\partial_0 W \to N\times \{0\}$ is the identity and 
$\partial_1 W \to N\times \{1\}$ is a homotopy equivalence and the surgery obstruction of $f$ rel $\partial W$ is equal to $\sigma$.
\end{thm}

\begin{proof}  By Corollary \ref{cor:all-agree}, it suffices to realize $\sigma$ as 
a Wall-type surgery obstruction.
The idea of the proof is to follow Wall's scheme
\cite[th.\ 5.8]{Wall}. The main difference here is that
intersection numbers are to be prescribed homotopy theoretically using the stable Hopf invariant of a Pontryagin-Thom construction.  

First assume that $\pi$ is trivial. Recall that a {\it $(-1)^k$-hermitian form} of rank $r$ over $\Bbb Z$
is a triple $(\Bbb Z^r,\lambda,\mu)$ in which $\lambda\: \Bbb Z^r \otimes \Bbb Z^r \to \Bbb Z$ is a non-singular $(-1)^k$-symmetric bilinear
form and $\mu\: \Bbb Z^r \to Q_k$ is a quadratic refinement of $\lambda$ \cite[p.~47-48]{Wall}.
\medskip

\noindent{\it Step one:}
By Wall's definition of $L_d$ \cite[p.~49]{Wall}, the element $\sigma$ is represented by a $(-1)^k$-hermitian form over $\Bbb Z$, 
say, of rank $r$. 

 Let $T \subset \text{int}(D^{2k-1})$ be a finite subset of cardinality $r$.
Let 
\[
P = T\times S^{k-1} \times D^k\, .
\] Then $P$ consists of $r$ disjoint copies of $S^{k-1} \times D^k$.
The goal of the first step is to find a ``trivial'' codimension zero Poincar\'e embedding of $P\to N$.

To achieve this, take $S^{k-1} \times D^k \subset \partial(D^k \times D^k) = S^{2k-1}$. By removing
a small open disk in $S^{2k-1}$,  one obtains a Poincar\'e embedding $S^{k-1} \times D^k \to D^{2k-1}$. Taking
 $r$ disjoint copies of the latter yields a Poincar\'e embedding $P\to T\times D^{2k-1}$. 
Then we have a Poincar\'e embedding $T\times D^{2k-1} \to D^{2k-1}$, 
where each component $\{t\} \times D^{2k-1} \to D^{2k-1}$ is identified with the restriction of the identity map
of a small disk containing the point $t$. We also have a Poincar\'e embedding $D^{2k-1} \to N$.\footnote{A Poincar\'e embedding  $D^{2k-1}\to N$ can be obtained from a top cell decomposition of $N$. 
Alternatively, one may use Theorem \ref{my_thm} to Poincar\'e embed  a point in $N$.}
We now take the composition of the Poincar\'e embeddings  $P \to T\times D^{2k-1} \to D^{2k-1} \to N$ to complete the step.
  \medskip

\noindent {\it Step two:} The next step is to modify the Poincar\'e embedding $P \to N$
to take into account  hermitian form data.

Let $E(P,N)$ denote the space of Poincar\'e embeddings of $P$ in $N$ (cf.\ \cite[ex.\ 2.15]{GK}). 
Let ${\cal I}(P,N)$ denote the
space of Poincar\'e immersions of $P$ in $N$. The latter can be defined as the space of Poincar\'e embeddings
$P \times D^j \to N \times D^j$ as $j$ tends to infinity  \cite{Klein_immersion}. The Pontryagin-Thom construction assigns to a
Poincar\'e  embedding its collapse map $N^+ \to P/\partial P$ which maps the fundamental class for 
$N$ to the fundamental class for $(P,\partial P)$ (here $N^+$ denotes $N$ with a disjoint basepoint).
Similarly, the Pontryagin-Thom construction associates to a Poincar\'e immersion a stable collapse map
$N^+ \to  P/\partial P$.
We obtain in this way a commutative square
\begin{equation}\label{eqn:0-cartesian}
\xymatrix{
E(P,N) \ar[r] \ar[d] & F(N^+,P/\partial P) \ar[d] \\
\cal I(P,N) \ar[r] & F(N^+,Q(P/\partial P))\, ,
}
\end{equation}
where the spaces on the right side are function spaces of based maps, and 
$Q = \Omega^\infty\Sigma^\infty$. If $k \ge 4$, then
the square  is $0$-Cartesian \cite[th.\ B]{Klein_compression}.
Using the given Poincar\'e embedding $P \to N$,  the square
becomes a diagram of based spaces. 
We will employ the stable EHP sequence for $P/\partial P$, i.e, 
\begin{equation} \label{eqn:EHP}
P/\partial P \overset{E}\to Q(P/\partial P) \overset {H}\to Q(D_2(P/\partial P))\, ,
\end{equation}
 \cite[thm.~1.11]{Milgram}, \cite{Crabb-Ranicki}, \cite[p.~653]{Goodwillie_calc3}, where $D_2(P/\partial P)$ is the quadratic construction of \eqref{eqn:quadratic-construction}.
The sequence \eqref{eqn:EHP} is a fibration in the metastable range, meaning that $H\circ E$ has a preferred null-homotopy and 
 the map from $P/\partial P$ to the homotopy fiber of the stable Hopf invariant $H$ is $(3k-2)$-connected. 
 
 One infers that there is a
$(k-1)$-connected map from the  homotopy fiber of the right vertical map of the square
to the function space $F(\Sigma (N^+),QD_2(P/\partial P))$. Consequently, if $k \ge 4$, we obtain a surjection
\begin{equation} \label{eqn:fiberwise-normal}
\pi_0(E\cal I(P,N)) \to \{\Sigma (N^+),D_2(P/\partial P)\}\, ,
\end{equation}
where $E\cal I(P,N)$ denotes the homotopy fiber of $\cal E(P,N) \to \cal I(P,N)$ at the basepoint,
and $\{\Sigma N^+,D_2(P/\partial P)\}$ is the abelian group of  homotopy classes of stable maps
$\Sigma (N^+)\to D_2(P/\partial P)$.
We will identify the target of \eqref{eqn:fiberwise-normal}.

By elementary obstruction theory, as $D_2(P/\partial P)$ is $(2k-1)$-connected,
 the target of \eqref{eqn:fiberwise-normal} is isomorphic to the cohomology group
\[
H^{2k-1}(N; \pi_{2k}(D_2(P/\partial P))).
\]
By Poincar\'e duality and the Hurewicz theorem, one has isomorphisms
\begin{align*}
H^{2k-1}(N; \pi_{2k}(D_2(P/\partial P))) & \cong H_0(N; \pi_{2k}(D_2(P/\partial P))) \, ,\\
&\cong \pi_{2k}(D_2(P/\partial P))\, , \\
&\cong H_{2k}(D_2(P/\partial P)) \, .
\end{align*}

Observe that $P/\partial P = \vee_T (S^{k-1}_+ \smsh S^k) \simeq \vee_T (S^{2k-1} \vee S^k)$. Let $\vee_T(S^k) \to P/\partial P$ be the inclusion.
Then $D_2(\vee_T S^k) \to D_2(P/\partial P)$ is $(3k-2)$-connected. Furthermore, if we identify $T$ with the ordered set $\{1,\dots , r\}$, then we obtain a splitting
\[
D_2(\underset{T}\vee S^k) \simeq \underset{i=1}{\overset{r}{\B{\bigvee}{.75}}} D_2(S^k) \,\, \vee \,\, \underset{i<j}{{\B{\bigvee}{.75}}}  S^{2k}\, .
\]
Taking homology, we find that the target of \eqref{eqn:fiberwise-normal} coincides with the abelian group
\[
H_{2k}(D_2(\underset{T}\vee S^k)) \,\,  \cong \,\,  \underset{i=1}{\overset{r}{\B{\bigoplus}{.75}}}H_{2k}(D_2(S^k)) \,\, \oplus \,\,  \underset{i<j}{{\B{\bigoplus}{.75}}} H_{2k}(S^{2k}) \\
\cong 
(\Bbb Z^r \otimes \Bbb Z^r)_{\Bbb Z_2}\, ,
\]
in which the last displayed term is the abelian   group of coinvariants of
 the involution on 
$\Bbb Z^r \otimes\Bbb Z^r$  given by
$(-1)^k\tau$, where $\tau$ is the map which switches factors.

Let $e_1,\dots,e_r\in \Bbb Z^r$ denote the standard basis.
Let $\phi\in (\Bbb Z^r \otimes \Bbb Z^r)_{\Bbb Z_2}$ be an element. Then any 
lift $\hat \phi \in \Bbb Z^r \otimes \Bbb Z^r$ of $\phi$
is necessarily of the form
\[
\sum_{ij} n_{ij} e_i \otimes e_j\, .
\]
Set $\lambda(e_i,e_j) = (n_{ij}+(-1)^kn_{ji})$. Then $\lambda$ extends bilinearly
to a $(-1)^k$-symmetric form $\Bbb Z^r\otimes \Bbb Z^r \to \Bbb Z$.
A quadratic refinement $\mu\: \Bbb Z^r \to Q_k$ of $\lambda$ is determined by setting
$\mu(e_i) = n_{ii}$. The condition that the form be non-singular is a non-singularity condition
on $\tr(\phi)\in \Bbb Z^r \otimes \Bbb Z^r$, where $\tr\: \Bbb Z^r\otimes_{\Bbb Z_2} \Bbb Z^r \to \Bbb Z^r\otimes \Bbb Z^r$
is the transfer. We have therefore extracted a  $(-1)^k$-hermitian form of rank $r$ over $\Bbb Z$ from the element
$\phi$, provided that $\tr(\phi)$ is non-singular. It is clear that the recipe defines a bijection between the elements
of  $\phi\in (\Bbb Z^r \otimes \Bbb Z^r)_{\Bbb Z_2}$ whose transfers  are non-singular
 and  $(-1)^k$-hermitian forms of rank $r$ over $\Bbb Z$.

We infer that  $(-1)^k$-hermitian forms of rank $r$ over $\Bbb Z$
are in bijection with a subset of $\{\Sigma N^+,D_2(P/\partial P)\}$.
Hence, if $(\Bbb Z^r,\lambda,\mu)$ is given, then
there is an element of $\pi_0(E\cal I(P,N))$ giving rise to it. An unraveling of the definition of $E \cal I(P,N)$ shows that
the element in question is represented by a Poincar\'e embedding of $P \times \{1\} \to N \times \{1\}$ together
with an extension of it to a Poincar\'e immersion of $P \times I \to N \times I$ that restricts to the basepoint
embedding $P \times \{0\} \to N \times \{0\}$.  Using the Poincar\'e embedding $P \times \{1\} \to N \times \{1\}$, we attach
 $r$ handles of index $k$ along $P$ to obtain
\[
W =  (N \times I) \cup_{P \times \{1\}} (\underset{r}\amalg(D^k \times D^k))\, .
\]
This defines the cobordism $W$. The rest of the argument follows
\cite[p.~55]{Wall}.
This completes the case when $\pi$ is trivial.

If $\pi$ is non-trivial, then the above argument may be modified
by pulling back along the universal cover $\tilde N \to N$
and working equivariantly.
We set
\[
Q_k := \Bbb Z[\pi]/(x - (-1)^k\bar x)\, ,
\]
where the involution $x \mapsto \bar x$ is  defined by
\[
\sum_{g\in \pi} n_g g \,\, \mapsto \,\, \sum_{g\in \pi} n_g w(g) g^{-1} \, ,
\]
where $w\: \pi \to \{\pm 1\}$ is the
orientation character of $N$.

The analogue of  \eqref{eqn:fiberwise-normal} is a surjective homomorphism
\[
\pi_0(E\cal I(P,N)) @>>> \{\Sigma (\tilde N^+), D_2(\tilde P/\partial \tilde P)\}_\pi  \cong 
H^{2k-1}(N;\pi_{2k}(D_2(\tilde P/\partial \tilde P)))
\]
(cf.~\cite[\S1.3]{Klein_compression}), where $\tilde P = P \times_{N} \tilde N$ and $\tilde \partial P = \partial P \times_{N} \tilde N$, and 
the displayed cohomology group is understood to be taken in the twisted sense
with respect to the $\Bbb Z[\pi]$-module $\pi_{2k}(D_2(\tilde P/\partial \tilde P))$.
By Poincar\'e duality, 
\[
H^{2k-1}(N;\pi_{2k}(D_2(\tilde P/\partial \tilde P))) \cong H_0(N;\Bbb Z^{w} \otimes \pi_{2k}(D_2(\tilde P/\partial \tilde P)))\, ,
\]
where on the right the coefficients are twisted by the orientation character.
Then arguing as in the 1-connected case, the homology group on the right is in bijection
with the set of  $(-1)^k$-hermitian forms of rank $r$ over $\Bbb Z[\pi]$. 
The rest of the argument is as in the $1$-connected case.
This sketches the  proof when $\pi$ is arbitrary. \end{proof}

\section{Poincar\'e thickenings}\label{sec:thick}

The purpose of this section is to provide some details
of the theory of Poincar\'e duality space thickenings in the stable range. The main result 
corresponds to what is classically known in the smooth case
\cite{Wall_thickening}. 
Corollary \ref{cor:thickening_stable_range} below is crucial to the proof of Proposition \ref{prop:t-onto}
and therefore to the proof of Theorem \ref{bigthm:stronger-exactness}.

Let $X$ be a homotopy finite space.
 A {\it Poincar\'e thickening} of $X$ of dimension $d$ is
a homotopy finite space $P$ equipped with a map $f\:P \to X$ such that
the following conditions are satisfied:
\begin{enumerate}[(i).]
\item (Duality). The pair $(\bar X,P)$ is a  Poincar\'e pair of
dimension $d$, where $\bar X$ is the mapping cylinder of $f$.
\item (Weak transversality). If $X$ has the homotopy type of a CW complex of dimension $\le s$, then the map $f$ is $(d-s-1)$-connected.
\end{enumerate}
In what follows, we abbreviate terminology and call  $f$ a {\it Poincar\'e $d$-thickening}.
We may think of $f$ as an object of $\Top_{/X}$. 

Two Poincar\'e thickenings $f_0 \: P_0 \to X$ and $f_1\: P_1 \to X$ 
of the same dimension are {\it concordant} if there is a finite chain
of weak equivalences between $f_0$ and $f_1$ in $\Top/X$.

\begin{rem} A variant definition of a Poincar\'e $d$-thickening of $X$ is to specify
 a triple $(V,\partial V, h)$ in which $(V,\partial V)$ is a $(d-s-1)$-connected
Poincar\'e pair of dimension $d$ and 
$h\: X @>>> V$ is a homotopy equivalence. This definition emphasizes the idea that 
a Poincar\'e thickening is a kind of ``Poincar\'e regular neighborhood'' of $X$.\footnote{This accounts for the weak transversality condition (ii), since in the manifold case
the map $\partial V \to V$ is $(d-s-1)$-connected by transversality.}
Up to a suitable notion of concordance, the two definitions coincide.
\end{rem}

Let $\cal T_d(X)$ be the set of concordance classes of Poincar\'e $d$-thickenings.
If $X$ and $Y$ are homotopy equivalent, then it is easily seen that
$\cal T_d(X) \cong \cal T_d(Y)$.
Unreduced fiberwise suspension defines a {\it stabilization
map} $S_X \: \cal T_d(X) \to \cal T_{d+1}(X)$.
Let $\cal T_\infty(X)$ be the colimit of $\cal T_d(X)$ with respect to $S_X$.
Then $\cal T_\infty(X)$ has a preferred basepoint, the {\it Euclidean thickening} of $X$, which is the unique
stable thickening having a trivial Spivak fibration \cite{Klein_immersion}.\footnote{By implementing the dualizing spectrum of \cite{Klein_dualizing}, the Euclidean thickening
of $X$ may be constructed without resorting to manifolds.}

Let $\xi$ be a $(k-1)$-spherical fibration over $X$. If $P \to X$ is a Poincar\'e $d$-thickening, 
then one can form $S^\xi_X P =$ the fiberwise join of $\xi$ with $P$ over $X$. This is given
by taking the homotopy pushout of the diagram 
\[
P \leftarrow S(\xi) \times_X P \to S(\xi)\, ,
\] where
$S(\xi)$ is the total space of $\xi$. Then $S^\xi_X P \to X$ is a Poincar\'e $(d+k)$-thickening.

Let $BG$ be the classifying space for stable spherical fibrations over $X$. Then 
the join operation gives $BG$ the structure of a group-like homotopy commutative
topological monoid, so $[X,BG]$ has the structure of an abelian group. The fiberwise
join defines an action
\begin{equation}\label{eqn:action}
[X,BG] \times \cal T_\infty(X) \to \cal T_\infty(X) \, .
\end{equation}

\begin{lem}[Compare {\cite[prop.~5.1]{Wall_thickening}}]  The orbit of the basepoint defines a bijection
\[
[X,BG] \,\, \cong \,\, \cal T_\infty(X)\, .
\]
\end{lem}

\begin{proof} The orbit of the basepoint of $\cal T_\infty(X)$  defines a function
of based sets
\[
s\: [X,BG] \to  \cal T_\infty(X) \, .
\]
In the other direction there is a function $r\: \cal T_\infty(X)  \to [X,BG]$ which sends a Poincar\'e thickening
$P\to X$ to the classifying map of the Spivak fibration of the associated Poincar\'e pair $(\bar X,P)$. 
Then $r\circ s(\xi) = -\xi$. It follows that $s$ is one-to-one.

Suppose that $P \to X$ is a Poincar\'e thickening with Spivak fibration $\xi$. Let $\eta$ be 
an unstable spherical fibration  representing $-\xi$, i.e., the fiberwise join $\eta \ast_X \xi $ is
fiber homotopically trivial. Then $\eta\ast_X \xi$  represents the basepoint of $[X,BG]$, and
 $s(\xi)$ is represented by $S_X^\xi S^{\eta}_X P \to X$, the latter  which is stably concordant to 
$P\to X$. It follows that $s$ is onto. 
\end{proof}

\begin{cor} The action \eqref{eqn:action} is free and transitive.
\end{cor}

\begin{proof} With respect to the identification defined by the bijection $s$,
the action corresponds to the operation
\[
[X,BG] \times [X,BG]  \to [X,BG]
\]
defined by $(\xi,\eta)\mapsto -(\xi+\eta)$. The latter is free and transitive.
\end{proof}

\begin{prop} \label{prop:suspension} Assume that $X$ has the homotopy type of a connected CW complex
of dimension $\le s$ with $d \ge \max(s+3,2s)$. Then 
\[
S_X\: \cal T_d(X) \to \cal T_{d+1}(X)
\]
is onto if $d\ge \max(s+3,2s)$ and is a bijection if $d \ge \max(s+3,2s+1)$.
\end{prop}

\begin{proof} Assume $d\ge \max(s+3,2s)$ and let
 $f\: P \to X$ represent an element of $\cal T_{d+1}(X)$. Then $f$
is $(d-s)$-connected. Consequently, $f$ has a section up to 
homotopy $g\: X \to P$. Then $g$ is $(d-s-1)$-connected. In particular, $g$ is $2$-connected.
By \cite[thm.~A]{Klein_haef} (or Theorem \ref{my_thm}), we infer that
$g$ is represented by a Poincar\'e embedding, i.e., there is a $(d+1)$-dimensional Poincar\'e
pair $(P_0,\partial P_0)$ and a  map $h:\partial P_0 \to X$ 
such that $g$ extends to a  homotopy equivalence
$X \cup_{\partial P_0} P_0\overset{\sim}\to P$.  The map $h$ then represents an element of 
$\cal T_d(X)$ which stabilizes to the element
represented by $f\: P \to X$. 

If $d \ge \max(s+3,2s+1)$, then the injectivity of $S$ is a straightforward consequence of 
 the main result of \cite{Klein_haef2}, which is a relative version of
the main result of \cite{Klein_haef}. 
\end{proof}

Let $BG_{d-s}$ denote the classifying space of $(d-s-1)$-spherical fibrations.
Then one has a diagram
\[
[X,BG_{d-s}] \to [X,BG] \cong \cal T_{\infty}(X) \leftarrow \cal T_d(X)\, ,
\]
so $[X,BG_{d-s}]$ and $\cal T_d(X)$ are approximations to 
$[X,BG]$.

It is well-known that the map $BG_{d-s} \to BG$ is $(d-s)$-connected, so by elementary
obstruction theory, the function $[X,BG_{d-s}] \to [X,BG]$ is surjective
when $d\ge 2s$ and is an isomorphism if $d \ge 2s+1$.

We also have, by Proposition \ref{prop:suspension},
 the following statement.

\begin{cor} \label{cor:thickening-spherical-fib} 
If $d \ge \max(s+3,2s)$,  then the function
\[
\cal T_d(X)\to  \cal T_{\infty}(X)
\]
is surjective. Moreover, if $d\ge \max(s+3,2s+1)$, then it is bijective.
\end{cor}

\begin{rem} If $X$ is a Poincar\'e  space of dimension $s$, then $\cal T_d(X) \cong [X,BG_{d-s}]$
when $d\ge s+3$. This follows from the observation that a Poincar\'e 
$d$-thickening $P \to X$ has the fiber homotopy type
of a $(d-s-1)$-spherical fibration in this case \cite[lem.~I.4.3]{Browder}, \cite[thm.~B]{Klein_fibration}. 
\end{rem}

Let $\xi$ be a stable spherical fibration over $X$, where $\dim X \le s$. 
If $d \ge \max(s+3,2s)$,
then Corollary \ref{cor:thickening-spherical-fib} implies that there is a Poincar\'e $(d+1)$-thickening 
$f\: Q\to X$ corresponding to $-\xi$ which is unique up to concordance. By unraveling the  correspondence, the Poincar\'e pair $(\bar X,Q)$ 
has Spivak fibration $\xi$. 

\begin{cor} \label{cor:thickening_stable_range} 
Assume $d \ge \max(s+3,2s)$. Then there is a Poincar\'e $(d+1)$-thickening 
$Q \to X$ with Spivak fibration $\xi$. If $d \ge \max(s+3,2s+1)$, the thickening 
is unique up to concordance.
\end{cor} 

\begin{cor}\label{cor:5-thickening} If $X$ has the homotopy type of a finite CW complex of
dimension $\le 2$, then there is a Poincar\'e 5-thickening $Q\to X$ with Spivak fibration $\xi$ which is
unique up to concordance.
\end{cor}


\end{document}